\documentclass[11pt]{amsart}   
\usepackage[english]{babel}
\usepackage{mathtools}
\usepackage{amssymb,amscd, graphics,color,latexsym,cancel}   
\usepackage[linktoc=all, pagebackref, hyperindex]{hyperref}%
\hypersetup{colorlinks,  citecolor=blue,  filecolor=blue,  linkcolor=blue,  urlcolor=black}

\usepackage{tikz-cd} 
\usepackage{extpfeil}

\usepackage{tikz}
\usetikzlibrary{matrix,calc}
\usepackage{mathrsfs}
\usepackage[all]{xy}
\usepackage{comment}
\usepackage{amsfonts}
\usepackage[normalem]{ulem}

\usepackage{euscript}
\usepackage{amsmath}
\usepackage{ifpdf}
\usepackage{cleveref}
\newcommand{\lar}{\longrightarrow}
\newcommand{\surjects}{\twoheadrightarrow}

\newtheorem{Theorem}{Theorem}[section]
\newtheorem{Lemma}[Theorem]{Lemma}
\newtheorem{Corollary}[Theorem]{Corollary}
\newtheorem{Proposition}[Theorem]{Proposition}

\theoremstyle{definition}
\newtheorem{Remark}[Theorem]{Remark}
\newtheorem{Example}[Theorem]{Example}

\newtheorem{Definition}[Theorem]{Definition}
\newtheorem{Question}[Theorem]{Question}
\def\sqr#1#2{{\vcenter{\hrule height.#2pt
			\hbox{\vrule width.#2pt height#1pt \kern#1pt
				\vrule width.#2pt}
			\hrule height.#2pt}}}
\def\phi{\varphi}

\def\VaVa{{\mathcal V}\kern-5pt {\mathcal V}}
\def\gr#1#2{{\rm gr}\, _{#1}(#2)}
\def\gr{{\rm gr}\,}
\def\hht{{\rm ht}\,}
\def\depth{{\rm depth}\,}

\def\Min{{\rm Min}\,}
\def\codim{{\rm codim}\,}

\def\ker{{\rm ker}\,}
\def\grade{{\rm grade}\,}
\def\rk{\rm rank}

\def\sym#1#2{\mbox{\rm Sym}_{#1}(#2)}

\def\Ext#1#2#3#4{{\rm Ext}\,^{#1}_{#2}({#3},{#4})}

\def\supp#1{{\rm Supp}\, (#1)}

\def\ini{\mbox{\rm in}}

\def\Rees{{\mathcal R}}
\def\sym{{\mathrm{Sym}}}
\def\cl#1{{\mathcal #1}}

\def\Rees{\mathcal{R}}

\def\phi{\varphi}
\def\Rees{{\cal R}}
\def\hht{{\rm ht}\,}
\def\grade{{\rm grade}\,}
\def\rt{{\rm rt}\,}

\def\fa{{\mathfrak a}}

\def\fm{{\mathfrak m}}
\def\fn{{\mathfrak n}}

\def\fp{{\mathfrak p}}

\def\fm{{\mathfrak m}}
\def\fn{{\mathfrak n}}

\def\cl#1{{\cal #1}}
\def\rk{\rm rank}

\newcommand{\excise}[1]{}

\def\NZQ{\mathbb}               

\def\PP{{\NZQ P}}

\def\G{{\mathcal G}}

\def\Q{{\mathcal Q}}
\def\opn#1#2{\def#1{\operatorname{#2}}} 
\opn\chara{char} \opn\length{\ell} \opn\pd{pd} \opn\rk{rk}
\opn\projdim{proj\,dim} \opn\injdim{inj\,dim} \opn\rank{rank}
\opn\depth{depth} \opn\grade{grade} \opn\height{height}
\opn\embdim{emb\,dim} \opn\codim{codim}
\def\OO{{\mathcal O}}
\opn\Tr{Tr} \opn\bigrank{big\,rank}
\opn\superheight{superheight}\opn\lcm{lcm}
\opn\trdeg{tr\,deg}
	\opn\reg{reg} \opn\lreg{lreg} \opn\ini{in} \opn\lpd{lpd}
	\opn\size{size} \opn\sdepth{sdepth}
	\opn\link{link}\opn\fdepth{fdepth}\opn\lex{lex}
	\opn\tr{tr}
	\opn\type{type}
	\opn\div{div} \opn\Div{Div} \opn\cl{cl} \opn\Cl{Cl}
	\opn\Spec{Spec} \opn\Supp{Supp} \opn\supp{supp} \opn\Sing{Sing}
	\opn\Ass{Ass} \opn\Min{Min}\opn\Mon{Mon}
	\opn\Ho{H}
	\opn\Ann{Ann} \opn\Rad{Rad} \opn\Soc{Soc}
	\opn\Im{Im} \opn\Ker{Ker} \opn\Coker{Coker} \opn\Am{Am}
	\opn\Hom{Hom} \opn\Tor{Tor} \opn\Ext{Ext} \opn\End{End}
	\opn\Aut{Aut} \opn\id{id}
	
	\opn\nat{nat}
	\opn\pff{pf}
	\opn\Pf{Pf} \opn\GL{GL} \opn\SL{SL} \opn\mod{mod} \opn\ord{ord}
	\opn\Gin{Gin} \opn\Hilb{Hilb}\opn\sort{sort}
	\opn\PF{PF}\opn\Ap{Ap}\opn\HF{HF}\opn\indeg{indeg}
	\opn\aff{aff} \opn
	\con{conv} \opn\relint{relint} \opn\st{st}
	\opn\lk{lk} \opn\cn{cn} \opn\core{core} \opn\vol{vol}  \opn\inp{inp} \opn\nilpot{nilpot}
	\opn\link{link} \opn\star{star}\opn\lex{lex}\opn\set{set}
	\opn\width{wd}
	\opn\Fr{F}
	\opn\QF{QF}
	\opn\G{G}
	\opn\type{type}\opn\res{res}
	\opn\log{Log}
	\opn\gr{gr}
	\def\Rees{{\mathcal R}}
	\def\pot#1#2{#1[\kern-0.28ex[#2]\kern-0.28ex]}

	\opn\dirlim{\underrightarrow{\lim}}
	\opn\inivlim{\underleftarrow{\lim}}
\begin{document}
		
		\title[ relation type of points]{The relation type of point configurations in the projective plane }
\author{Ethan Cotterill}
\address{Universidade Estadual de Campinas (UNICAMP) \\ Instituto de Matem\'atica, Estat\'{\i}stica e Computação Cient\'{\i}fica (IMECC) \\ Departamento de Matem\'atica \\ 	Rua S\'ergio Buarque de Holanda, 651\\ 13083-970 Campinas-SP, Brazil}
\email{cotterill.ethan@gmail.com}
\author{Amir Mohammad Kach khaali}
\address{Department of Mathematics, Institute for Advanced Studies in Basic Sciences (IASBS), Zanjan 45137-66731, Iran}
\email{amirmohammadkckh@iasbs.ac.ir}
\author{Abbas Nasrollah Nejad}
\address{Department of Mathematics, Institute for Advanced Studies in Basic Sciences (IASBS), Zanjan 45137-66731, Iran}
\email{abbasnn@iasbs.ac.ir}

\begin{abstract}
We study the {\it relation type} of ideals of finite reduced sets of points in the projective plane; 
for a given ideal, this is the maximal $T$-degree of a minimal generator of the Rees algebra. Our main focus is on point configurations whose defining ideals are not necessarily
linearly presented, 
with an emphasis on almost collinear configurations. We prove that $\rt(X)\in\{1,3\}$ whenever $X\subseteq \PP^2_k$ is a finite set of at most ten points; and we characterize the configurations of relation type $3$ in this range. We then show that a configuration of eleven points in generic position has relation type $5$,
thereby yielding the first occurrence of relation type larger than $3$. Finally, we exhibit a configuration of $17$ points with relation type $4$; 
and we formulate some questions regarding the spectrum 
of admissible relation types of point configurations.
\end{abstract}
		
\maketitle
\section*{Introduction}

Let $I=(f_1,\ldots,f_m)$ be an ideal of a Noetherian ring $R$. The Rees algebra of $I$ is
\[
\Rees_R(I)=\bigoplus_{n\geq 0}I^nt^n=R[It]\subseteq R[t].
\]
Algebraically, it records all powers of $I$  simultaneously; geometrically, it is the homogeneous coordinate ring of the blowup of the ambient space along the subscheme defined by $I$. Rees algebras and their defining equations have been studied by many authors; see, for instance~\cite{BCS,Her,SUV,UV93,Vas91,VasBook}. 

\medskip
Choosing generators of $I$ yields a standard graded presentation
\[
S=R[T_1,\ldots,T_m]\lar \Rees_R(I),\qquad T_i\longmapsto f_it.
\]
Let $\Q$ denote its kernel. 
The \emph{relation type} of
$I$, denoted $\rt_R(I)$, is the largest $T$-degree of a minimal homogeneous generator of
$\Q$; equivalently, it is the least integer $\ell$ such that 
the blowup of $\mbox{Proj}(R)$ along $I$ is 
cut out by equations of $T$-degree at most $\ell$.

\medskip
Thus relation type measures the complexity of the equations defining the blowup. The
case $\rt_R(I)=1$ is the classical case of linear type: the symmetric algebra
$\operatorname{Sym}_R(I)$ already coincides with the Rees algebra, or equivalently the
linear equations arising from the syzygies of $I$ define the blowup scheme-theoretically.
Higher relation type detects genuinely nonlinear equations.

\medskip
Relation type is also closely connected with Andr\'e--Quillen homology and with the
module of effective relations; see \cite{andre,quillen1,FV1}. This homological
viewpoint shows that relation type is an intrinsic invariant in several natural
categories. Planas-Vilanova proved that relation type is an invariant of affine
algebraic varieties \cite{FV-RT}; by \cite{ANM}, it is also an invariant of analytic
and formal germs. For a projective scheme $X\subseteq \PP^n_k$ with saturated
homogeneous defining ideal $I_X\subseteq A=k[x_0,\ldots,x_n]$, we define
\[
\rt(X):=\rt_A(I_X)
\]
to be the relation type of the ideal of the affine cone over $X$.

\medskip
Throughout this paper, $k$ is an algebraically closed field of characteristic zero and $R=k[x,y,z]$. We study $\rt(X)$ for finite  sets of distinct points $X\subseteq \PP^2_k$. The defining ideal $I_X$ is saturated of height two, hence perfect, and its Hilbert--Burch resolution is strongly influenced by the incidence geometry of $X$. This makes point sets in the plane a natural class in which to study how equations of blowups reflect elementary geometry.

\medskip
A fundamental 
point of departure is the first gap theorem for ideals of points in the plane. According to
\cite{ANM}, recalled below as Proposition~\ref{prop:points-plane-gap}, 
we have
\[
\rt(X)=1\qquad\text{or}\qquad \rt(X)\geq 3
\]
for every finite set of points $X\subseteq \PP^2_k$.
In other words, 
2 is never the relation type of a configuration of points in the plane. This is because whenever $\mu(I_X)\leq 3$, $I_X$ is either a complete intersection or a strict almost complete intersection; since $I_X$ is locally a complete intersection on the punctured spectrum, 
it follows that
$\rt(X)=1$. On the other hand, if $\mu(I_X)\geq 4$, then $I_X$ is not of linear type by the
Herzog--Simis--Vasconcelos generator bound \cite[Proposition~2.4]{HSV}. Moreover, the
Andr\'e--Quillen description of relation type, together with
\cite[Remark~2.9]{HSV}, shows that $\rt(I_X)\leq 2$ would force $\rt(I_X)=1$.

\medskip
When $I_X$ is linearly presented and $\mu(I_X)\geq 4$, the Morey--Ulrich theorem gives
the expected equations of the Rees algebra 
in terms of the maximal minors of a Jacobian dual
matrix; see \cite[Theorem~1.3]{MU}. Hence the first gap immediately gives
$\rt(X)=3$. The main difficulty begins when $I_X$ is not linearly presented. This is the
problem suggested in \cite[Example~3.7]{ANM}: 
to compute the relation type of non-linearly presented ideals of point configurations.
Here we solve this problem
for large families of 
almost collinear configurations and for all point sets of small
cardinality.

\medskip
A first source of relation type one comes from low-degree containment. If $X$ is
collinear, then $I_X$ is a complete intersection. More generally, by Burch's theorem, in
the form stated in \cite[Corollary~3.8]{EisSyz}, if a finite set of points in
$\PP^2_k$ lies on a curve of degree $d$, then its defining ideal is generated by at most
$d+1$ elements. In particular, if $X$ is contained in a conic, then $I_X$ is generated by
at most three elements, and therefore $\rt(X)=1$; see
Proposition~\ref{prop:points-on-conic}. Thus the first interesting configurations are
those not contained in a conic, especially when their defining ideals are not linearly
presented.

\medskip
The main technical 
content of the paper concerns $(s-r)$-fold collinear configurations,
introduced in \cite{AZ}. Such a configuration is a set
\[
X=Y\cup Z\subseteq \PP^2_k,
\]
where $Y$ consists of $s-r$ points on a distinguished line $L$ and $Z$ consists of the
remaining $r$ points off $L$. If $L=V(\ell)$ and there exists a form
$F\in(I_Z)_{|Y|}$ whose image modulo $\ell$ cuts out $Y$ on $L$, then
\[
I_X=\ell I_Z+(F).
\]
This basic double-link description 
leads to an explicit Hilbert--Burch resolution of
$I_X$ in terms of one for $I_Z$, and it is the central construction used throughout the
paper.

\medskip
Using this construction together with Jacobian-dual matrices and some additional structural criteria for 
Rees algebras, we prove several structural results for 
$(s-r)$-fold collinear configurations. The first treats the case where the residual block is itself almost collinear: if $Z$ has $r-1$ points on a line and one point off that line, then $\rt(X)=1$ 
in exactly one exceptional geometric
case; in all other cases $\rt(X)=3$; see
Theorem~\ref{thm:rt-residual-r-1-fold}. The second treats triangular residual blocks in
{\it generic 
position}: if $|Z|=\binom{d+1}{2}$, $Z$ has generic (i.e., maximum possible) Hilbert function, and the
collinear block has cardinality at least $d+1$, then $\rt(X)=3$; see
Theorem~\ref{thm:rt-triangular-block}.

\medskip
We then obtain explicit classifications for the first nontrivial residual cardinalities, corresponding to $(s-4)$-, $(s-5)$-, and $(s-6)$-fold collinear configurations. For $r\leq 6$,
the Hilbert function and the generators of $I_Z$ are controlled by elementary incidence
conditions relative to lines and conics. This is also why 
our 
complete incidence-by-incidence
classification 
stops at $r=6$: at $r=7$ and above, linear systems of cubics passing through the
residual block $Z$ enter in an essential way and the relation type is no longer determined by
line and conic incidences alone.

\medskip
The first main global consequence of our structural results is the complete characterization of possible relation types for point configurations of small cardinality. In
Theorem~\ref{thm:rt-up-to-ten-points}, we prove that for every finite set of distinct points $X\subseteq \PP^2_k$ with $4\leq |X|\leq 10$, we have
\[
\rt(X)\in\{1,3\}.
\]
Moreover, for every cardinality $6\leq |X|\leq 10$, we characterize precisely those point
configurations with relation type $3$. The proof combines the almost collinear
classifications with Hilbert function arguments and 
specific implementations of the
Jacobian dual method.

\medskip
In particular, $\rt(X) \leq 3$ whenever $|X| \leq 10$; and this characterization of relation types in small cardinalities is sharp. Namely, in Theorem~\ref{thm:eleven-generic-position-rt-five}, we show
that eleven points in general position have relation type
\[
\rt(X)=5.
\]
Indeed, for eleven general points, $I_X$ is generated by four quartics and has Hilbert--Burch column degrees $(1,1,2)$. By \cite[Corollary~C]{CidRuiz}, the associated rational map satisfies
\[
\deg(\phi)\deg_{\PP^3}(Y)=e_2(1,1,2)=1\cdot 1+1\cdot 2+1\cdot 2=5.
\]
Since the rational map's image is not a plane, we have $\deg_{\PP^3}(Y)>1$. Hence $\deg(\phi)=1$ and $\deg_{\PP^3}(Y)=5$. Thus the special fiber has an implicit equation of degree $5$, and this equation lifts to a Rees equation of $T$-degree $5$. Conversely, the saturation description of the defining ideal of the Rees algebra, together with Jouanolou's inertia-form
bound for bidegrees $(1,1),(1,1),(2,1)$ \cite[\S3.11.19]{JouanolouInertia}, shows that
no equation of higher $T$-degree is needed.

\medskip
Consequently, relation type strictly greater than $3$ first manifests at cardinality eleven. 
This, in turn, motivates the spectrum problem for eleven points,
formulated as Question~\ref{ques:eleven-points-spectrum}: is every set of eleven points
of relation type $1$, $3$, or $5$?

\medskip
The latter problem requires assuming the configuration be specifically of cardinality 11. 
Indeed, in Example~\ref{ex:seventeen-points-rt-four}, we give a reduced set of
$17$ points whose defining ideal is generated by four quintics, has Hilbert--Burch
column degrees $(1,2,2)$, and has relation type $4$. The 
reason comes from the
special fiber: for column degrees $(1,2,2)$ the degree formula yields
\[
1\cdot 2+1\cdot 2+2\cdot 2=8.
\]
When the associated rational map has degree $2$ onto its image, the image is a quartic
surface, producing an essential Rees equation of $T$-degree $4$.

\medskip
We close the paper by discussing the basic configurations $B_d$ of Geramita--Maroscia, which have
generic Hilbert function and equigenerated defining ideals but Hilbert--Burch matrices
with quadratic entries; see \cite{GM}. These provide natural test cases beyond the
linearly presented setting. In particular, computations show that $\rt(B_6)=12$. We
single out the determination of possible relation types of reduced
point configurations of fixed cardinality in general, and of the
Geramita--Maroscia configurations $B_d$ in particular, as problems warranting further attention.

\section{Relation Type and the Geometry of Points}
\label{sec:rt-geometry-points}

In this section we review the notion of relation type for Rees algebras and explain how, via affine cones, it gives an invariant of projective schemes. 
We then specialize to
finite sets of points in $\PP^2_k$.

\medskip
Let $A$ be a Noetherian ring and let $I=(f_1,\ldots,f_m)$ be an ideal of $A$. Consider the standard graded presentation
\[
\Phi:S=A[T_1,\ldots,T_m]\lar \Rees_A(I),\qquad T_i\longmapsto f_i t
\]
in which the grading is the $T$-grading; so $\deg T_i=1$, and elements of $A$ have degree
$0$. Let $\Q:=\ker\Phi$ be the defining ideal of the Rees algebra. For $\ell\geq 1$,
let $\Q\langle \ell\rangle$ denote the ideal of $S$ generated by all homogeneous
elements of $\Q$ of $T$-degree at most $\ell$.

\begin{Definition}\label{def:relation-type-ideal}
The \emph{relation type} of $I$, denoted $\rt_A(I)$, is the least integer $\ell\geq 1$
for which $\Q=\Q\langle \ell\rangle$. Equivalently, $\rt_A(I)$ is the largest
$T$-degree of a minimal homogeneous generator of the defining ideal $\Q$ of
$\Rees_A(I)$.
\end{Definition}
Let $\sym_A(I)$ 
denote the symmetric algebra of $I$. Then  
$\sym_A(I)\simeq S/\Q\langle 1\rangle$,
and the natural map $\operatorname{Sym}_A(I)\surjects \Rees_A(I)$ has kernel
$\Q/\Q\langle 1\rangle$. So $\rt_A(I)=1$ if and only if this map is an isomorphism.
In this case, $I$ is said to be of \emph{linear type}. In particular, ideals generated by regular sequences are of linear type.

\medskip
Relation type also admits an interpretation in terms of Andr\'e--Quillen homology; namely:
$$\rt_R(I)=\min \{r\geq 1\ \colon \Ho_2(\Rees_R(I),R,R)_n=\Ho_1(R,\Rees_R(I),R)_n=0\ \mbox{for all }\ n\geq r+1   \}$$
	where $\Ho_2(\Rees_R(I),R,R)_n$ and $\Ho_1(R,\Rees_R(I),R)_n$ are the $n$-th graded component of the first and second Andr\'e-Quillen homologies; see~\cite[Definition 1.1]{ANM}. It follows that $\rt_R(I)$ is independent of the presentation of $\Rees_R(I)$. This homological viewpoint shows that relation type is
intrinsic in several natural categories; see \cite[Theorem~2.3 ]{ANM}.

\medskip
If $(A,\fm)$ is local with residue field $k=A/\fm$, or if $A$ is a standard graded
$k$-algebra with irrelevant maximal ideal $\fm$, the {\it special fiber of} (the Rees algebra of) $I$ is
\[
\mathcal F(I):=\Rees_A(I)\otimes_A A/\fm
\simeq k[T_1,\ldots,T_m]/\mathcal K,
\qquad
\mathcal K=\frac{\Q+\fm S}{\fm S}.
\]
By Nakayama's lemma, every essential equation of the special fiber lifts to
an equation of the Rees algebra, and its degree gives a lower bound for the relation type. The converse need not hold: an essential equation of $\Q$ may vanish modulo
$\fm$. Such equations are torsion equations of the symmetric algebra. If $A=k[x_0,\ldots,x_n]$ and $I$ is generated by homogeneous forms of the same degree,
then $\mathcal F(I)$ is the homogeneous coordinate ring of the image of the rational map
\[
\PP^n_k\dashrightarrow \PP^{m-1}_k,\qquad
p\longmapsto [f_1(p):\cdots:f_m(p)].
\]
Thus the special fiber is the projective image of the linear system generated by $f_1,\ldots,f_m$, whereas the Rees algebra describes the 
blowup of $\PP^n_k$ along the corresponding ideal.

\medskip
We also recall the {\it Jacobian dual} construction and the notion of expected equations. Let
\[
A^q \stackrel{\varphi}{\lar} A^m \lar I\lar 0
\]
be a presentation of $I=(f_1,\ldots,f_m)$, and set
$T=[T_1\ \cdots\ T_m]$. As above, the defining ideal of the symmetric algebra of $I$ is
\[
\mathcal L=I_1(T\varphi)=\Q\langle 1\rangle\subseteq S.
\]
Let $\fa=(a_1,\ldots,a_s)\subseteq A$ be an ideal containing the entries of $\varphi$.
Since the entries of $T\varphi$ belong to $\fa S$, we have
\[
T\varphi=[a_1\ \cdots\ a_s]B(\varphi)
\]
for some $s\times q$ matrix $B(\varphi)$ with entries in $S$, homogeneous of
$T$-degree $1$. We say $B(\varphi)$ is a \emph{Jacobian dual} of $\varphi$ with respect
to the sequence $a_1,\ldots,a_s$. In general, $B(\varphi)$ is not unique.

\medskip
The Jacobian dual is a natural source of higher-degree equations of the Rees algebra.
Indeed, by Cramer's rule, if the image of $\fa$ in $\Rees_A(I)$ contains a nonzerodivisor,
then
\[
\mathcal L+I_s(B(\varphi))\subseteq \Q .
\]
When equality holds, $\Q=\mathcal L+I_s(B(\varphi))$,
we say that the defining ideal of the Rees algebra has the \emph{expected form}. Thus, in
the expected form, the defining ideal of the Rees algebra is generated by the linear equations defining the
symmetric algebra together with the maximal minors of a Jacobian dual matrix.

\medskip
Recall that an ideal $I\subseteq A$ {\it satisfies $G_s$} if
$\mu(I_\fp)\leq \dim A_\fp$ for every $\fp\in V(I)$ with
$\dim A_\fp\leq s-1$. A fundamental theorem of
Morey--Ulrich states that if $A=k[x_1,\ldots,x_d]$ is a polynomial ring over a field,
$I\subseteq A$ is a perfect ideal of height two, $I$ is linearly presented, and $I$ satisfies
$G_d$, then the defining ideal of the Rees algebra has the expected form
$\Q=\mathcal L+I_d(B(\varphi))$. 
Moreover, $\Rees_A(I)$ is Cohen--Macaulay; see \cite[Theorem~1.3]{MU}.
We also mention that several weakenings of the Morey--Ulrich hypotheses have been studied,
including weakenings of the $G_d$ condition and of the linear-presentation assumption;
see \cite{BoswellMukundan,FRS,Lan2014}.

\medskip
These results have important implications for the relation type. 
For example, whenever the defining ideal of the Rees algebra has the expected form, the relation type is controlled by the size of the Jacobian dual minors. On the other hand, when the defining ideal is not of expected form, additional equations may appear, often from the special fiber or from torsion in the symmetric algebra, and these equations may increase the relation type. 
In this paper, we use the geometry of the special fiber, Jacobian dual equations, and the saturation of the symmetric algebra to detect the possible degrees of essential equations of the Rees algebra.

\medskip
We shall use the following invariance theorem for the relation type of $k$-algebras.
\begin{Theorem}[\cite{ANM,FV-RT}]\label{thm:quotient-invariance}
Let $k$ be a field. Let $A$ and $B$ be polynomial rings over $k$, and let
$I\subseteq A$ and $J\subseteq B$ be ideals. If $A/I\simeq B/J$ as $k$-algebras, then
$\rt_A(I)=\rt_B(J)$. 
\end{Theorem}
We now recall how relation type is attached to affine and projective schemes.
\begin{Definition}\label{def:affine-scheme-rt}
Let $X$ be an affine scheme of finite type over a field $k$. The \emph{relation type} of $X$ is defined by
\[
\rt(X):=\rt(\OO_X(X)).
\]
More generally, if $X$ is a scheme of finite type over $k$, then
\[
\rt(X):=\max\{\rt(\OO_{X,x}) \mid x\in X\}.
\]
\end{Definition}
Since the category of affine schemes of finite type over $k$ is anti-equivalent to the category of finitely generated $k$-algebras, Theorem~\ref{thm:quotient-invariance} shows that the relation type of an affine $k$-scheme is well-defined and invariant under affine isomorphisms. Moreover, since relation type is local~\cite[Proposition~1.3]{ANM}, the second definition is intrinsic and gives an invariant for schemes of finite type over $k$.

\medskip        
We now pass to projective schemes.
\begin{Definition}\label{def:projective-scheme-rt}
Let $X\subseteq \PP^n_k$ be a projective subscheme with saturated homogeneous defining
ideal $I_X\subseteq A=k[x_0,\ldots,x_n]$. We define the \emph{relation type} of $X$ by
\[
\rt(X):=\rt_A(I_X).
\]
Equivalently, $\rt(X)$ is the relation type of the affine cone over $X$.
\end{Definition}

Thus the relation type is understood with respect to a given projective embedding.
This is the natural viewpoint in the present paper, since our main objects are finite
sets of points in $\PP^2_k$ and their saturated homogeneous defining ideals.

\begin{Remark}\label{rem:ci-linear-type}
If $I_X$ is a complete intersection, then $\rt(X)=1$. More generally, if the affine
cone over $X$ is locally a complete intersection on the punctured spectrum and $I_X$ is
a strict almost complete intersection, then $\rt(X)=1$ by
\cite[Proposition~3.4]{ANM}.
\end{Remark}

\begin{Proposition}\label{prop:projective-invariance}
Let $X,Y\subseteq \PP^n_k$ be projective subschemes. If $X$ and $Y$ are projectively
equivalent, then $\rt(X)=\rt(Y)$.
\end{Proposition}

\begin{proof}
Let $g\in \operatorname{PGL}_{n+1}(k)$ with $Y=g(X)$, and choose a representative
$\widetilde g\in \operatorname{GL}_{n+1}(k)$. This gives a graded $k$-algebra
automorphism $\sigma:k[x_0,\ldots,x_n]\lar k[x_0,\ldots,x_n]$ such that, after replacing
$\sigma$ by its inverse if necessary, $I_Y=\sigma(I_X)$. Hence
$k[x_0,\ldots,x_n]/I_X\simeq k[x_0,\ldots,x_n]/I_Y$ as graded $k$-algebras. The result
follows from Theorem~\ref{thm:quotient-invariance}.
\end{proof}

\begin{Remark}\label{rem:projective-invariance-use}
By Proposition~\ref{prop:projective-invariance}, we may freely apply projective changes
of coordinates when studying relation type. For example, a line containing many points
may be assumed to be $V(z)$.
\end{Remark}

We now specialize to finite sets of distinct points in $\PP^2_k$. Given a finite set
$X\subseteq \PP^2_k$, let $I_X\subseteq R=k[x,y,z]$ denote its homogeneous defining ideal. Then
\[
I_X=\bigcap_{p\in X} I_p
\]
and each $I_p$ is generated by two independent linear forms. Hence $I_X$ is a
saturated homogeneous ideal of height $2$, and $R/I_X$ is a reduced graded one-dimensional
Cohen--Macaulay ring.

\medskip
 For a finite reduced set of points
$X\subseteq \PP^2_k$, the ideal $I_X$ satisfies $G_3$. Indeed, the minimal primes of
$I_X$ are the ideals of points $I_p$, each generated by two independent linear forms. If
$\fp\in V(I_X)$ and $\dim R_\fp\leq 2$, then $\fp$ is one of these point ideals.
Therefore $(I_X)_\fp=(I_p)_\fp$ is generated by two elements, and
$\mu((I_X)_\fp)=2=\dim R_\fp$. In particular, $I_X$ is locally a complete intersection on the punctured spectrum.
Hence it is of linear type away from $\fm$. Letting
$\mathcal L=\Q\langle 1\rangle$ denotes the defining ideal of the symmetric algebra, we have
\[
\Q=\mathcal L:(\fm S)^\infty .
\]
Indeed, the kernel of the natural map
$\operatorname{Sym}_R(I)\to \Rees_R(I)$ is $\Q/\mathcal L$. Since $I$ is of linear type
at every prime $\fp\neq\fm$, this module is supported along $V(\fm S)$, and hence
$\Q\subseteq \mathcal L:\fm^\infty$. Conversely, $\Rees_R(I)\subseteq R[t]$ is a
domain, so $\Q$ is prime and $\fm S\not\subseteq \Q$. If
$H\in \mathcal L:\fm^\infty$, then $(\fm S)^N H\subseteq \mathcal L\subseteq \Q$ for
some $N$. Since $(\fm S)^N\not\subseteq\Q$, primeness gives $H\in\Q$. Therefore the
nonlinear equations of the Rees algebra are obtained by saturating the linear equations:
some survive modulo $\fm$ and define the special fiber, while the others disappear
modulo $\fm$ and represent torsion equations of the symmetric algebra.

\medskip
We now turn from the general structure of Rees equations to the relation type of finite sets of points in $\PP^2_k$. We begin by recording some basic geometric situations where the relation type is forced to be $1$.
\begin{Example}\label{ex:collinear-points}
Let $X\subseteq \PP^2_k$ be a set of $s\geq 2$ distinct collinear points. After a
projective change of coordinates, assume that $X\subseteq V(z)$. Then $I_X=(z,f(x,y))$, where $f(x,y)$ is a reduced product of $s$ pairwise non-proportional linear forms. Thus
$I_X$ is a complete intersection, and therefore $\rt(X)=1$.
\end{Example}

We first record a simple consequence of Burch's theorem which will be used repeatedly.

\begin{Proposition}\label{prop:points-on-conic}
Let $X\subseteq \PP^2_k$ be a finite set of distinct points. If $X$ is contained in a
conic, then $\rt(X)=1$.
\end{Proposition}

\begin{proof}
By Burch's theorem, in the form stated in \cite[Corollary~3.8]{EisSyz}, if a finite set
of points in $\PP^2_k$ lies on a curve of degree $d$, then its defining ideal is
generated by at most $d+1$ elements. Since $X$ is contained in a conic, we have
$\mu(I_X)\leq 3$. As $\hht I_X=2$, the ideal $I_X$ is either a complete intersection or
a strict almost complete intersection. Hence Remark~\ref{rem:ci-linear-type} gives
$\rt(X)=1$.
\end{proof}
\begin{Remark}\label{rem:not-linear-type-many-generators}
We will also make repeated use of the following observation: if $X\subseteq \PP^2_k$ is a finite
set of distinct points and $\mu(I_X)\geq 4$, then $I_X$ is not of linear type. Indeed,
if $I_X$ were of linear type, then by \cite[Proposition~2.4]{HSV}, for every prime
$\fp\supseteq I_X$, the localization $(I_X)_\fp$ could be generated by $\hht(\fp)$
elements. Taking $\fp=\fm=(x,y,z)$ gives a contradiction, since
$\mu_{R_\fm}((I_X)_\fm)=\mu(I_X)\geq 4>3=\hht(\fm)$.
\end{Remark}

Remark~\ref{rem:ci-linear-type}, Example~\ref{ex:collinear-points} and Proposition~\ref{prop:points-on-conic} give configurations whose defining ideals are of linear type. 
However, the following {\it gap phenomenon} is operative. 
Namely, for a finite set of points in $\PP^2_k$, the first genuine obstruction to linear type cannot occur in $T$-degree $2$; the first possible nonlinear equations of the blowup appear in $T$-degree at least $3$. The following result, proved in \cite[Proposition~3.6]{ANM}, makes this precise. We include the proof for completeness,
since this gap is one of the 
points of departure for the current paper.

\begin{Proposition}\label{prop:points-plane-gap}
Let $X\subseteq \PP^2_k$ be a finite set of distinct points. Then either $\rt(X)=1$ or
$\rt(X)\geq 3$. Moreover, if $I_X$ is linearly presented and $\mu(I_X)\geq 4$, then
$\rt(X)=3$.
\end{Proposition}
\begin{proof}
Set $I=I_X$. If $\mu(I)\leq 3$, then $I$ is either a complete intersection or a strict almost complete intersection. Since, as explained above, $I$ is locally a complete intersection on the punctured spectrum, Remark~\ref{rem:ci-linear-type} gives
$\rt_R(I)=1$.

We may therefore assume that $\mu(I)\geq 4$. By
Remark~\ref{rem:not-linear-type-many-generators}, $I$ is not of linear type. Hence $\rt_R(I)\neq 1$. We claim that $\rt_R(I)\neq 2$. Suppose, to the contrary, that
$\rt_R(I)\leq 2$. The Andr\'e--Quillen characterization of relation type recalled above gives
$$\rt_R(I)=\min \{ r\geq 1 \ | \ \Ho_{1}(R,\Rees_{R}(I),R )_n=0\ \hbox{for all }\ n\geq r+1  \}$$ and 
it follows that $\Ho_{1}(R,\Rees_{R}(I),R )_n=0$ for all $n\geq 3$. By~\cite[Remark 2.9]{HSV}, we conclude that   $\Ho_{1}(R,\Rees_{R}(I),R )_2=0$, which implies that $\rt(I)=1$, a contradiction. Thus, $\rt(X)\geq 3$. 

\medskip
 For the second assertion, assume that
$I_X$ is linearly presented and $\mu(I_X)\geq 4$.  As explained above, $I_X$ satisfies $G_3$. Hence the Morey--Ulrich theorem gives $\Q=\mathcal L+I_3(B(\varphi))$, where $B(\varphi)$ is a Jacobian dual matrix; see
\cite[Theorem~1.3]{MU}. Therefore $1<\rt(X)\leq 3$. Since relation type $2$ does not occur, we conclude that $\rt(X)=3$.
\end{proof}

The next examples illustrate Proposition~\ref{prop:points-plane-gap}.

\begin{Example}\label{ex:star-configuration-points}
Let $L_1,\ldots,L_q\in R_1$, with $q\geq 4$, be pairwise non-proportional linear forms
such that no three of the lines $V(L_i)$ are concurrent. Let
\[
X=\{\,V(L_i,L_j)\mid 1\leq i<j\leq q\,\}\subseteq \PP^2_k
\]
be the corresponding star configuration of points. Set
\[
F=L_1\cdots L_q,\qquad M_i=\frac{F}{L_i}\quad\text{for }i=1,\ldots,q.
\]
Then $I_X=(M_1,\ldots,M_q)$, and $I_X$ has a Hilbert--Burch resolution
\[
0\lar R(-q)^{q-1}\lar R(-(q-1))^q\lar I_X\lar 0.
\]
Hence $I_X$ is linearly presented. Since $\mu(I_X)=q\geq 4$,
Proposition~\ref{prop:points-plane-gap} gives $\rt(X)=3$.
\end{Example}

\begin{Example}\label{ex:triangular-number-general-points}
Let $d\geq 3$ and let $X\subseteq \PP^2_k$ be a set of
$s=\binom{d+1}{2}$ distinct points such that $(I_X)_{d-1}=0$. Then $X$ has Hilbert
function
\[
\HF_{R/I_X}(j)=\min\left\{s,\binom{j+2}{2}\right\}.
\]
In particular,
\[
\mathrm{HS}_{R/I_X}(t)=\frac{1-(d+1)t^d+dt^{d+1}}{(1-t)^3}.
\]
Since $I_X$ is saturated of height $2$, it is perfect, and its minimal free resolution is
\[
0\lar R(-(d+1))^d\lar R(-d)^{d+1}\lar I_X\lar 0.
\]
Thus $I_X$ is generated by $d+1$ forms of degree $d$ and has a linear presentation.
Since $d+1\geq 4$, Proposition~\ref{prop:points-plane-gap} gives $\rt(X)=3$.
\end{Example}

We also recall a related asymptotic perspective. If $I_X$ is the defining ideal of a
finite set of points $X\subseteq \PP^2$, H.~T.~H\`a studied the Rees algebras of the ideals
generated by the homogeneous pieces $(I_X)_t$. These Rees algebras describe the graphs
of the rational maps defined by the linear systems $(I_X)_t$, and for $t$ sufficiently
large they give embeddings of the blowup of $\PP^2$ along $X$; see
\cite{Ha-Rees-points}. This viewpoint is different from ours: here we study the
relation type of the full homogeneous ideal $I_X$, not of a single graded piece.

\medskip

Thus, for finite sets of points in $\PP^2_k$ (whose homogeneous ideals are) not of linear type, the 
relation type is always $3$ or greater. When $I_X$ is linearly presented and $\mu(I_X)\geq 4$,
Proposition~\ref{prop:points-plane-gap} gives an effective criterion for $\rt(X)=3$ in terms of
the expected equations of the Rees algebra. The more subtle case is when $I_X$ is not
linearly presented. The goal of the next sections is to shed light on this case 
by focusing on special geometric configurations, beginning with almost collinear configurations.
\section{Relation Types of $(s-r)$-fold Collinear Configurations}
\label{sec:s-r-fold}

In this section we study finite sets of points in $\PP^2_k$ with a distinguished collinear subset. These configurations were introduced in \cite{AZ}. They provide a natural class beyond the linear-type cases considered above, and 
in particular, they serve as a testing ground for possible relation types $>1$. The results obtained here will be used later in the study of finite sets of points of small cardinality.

\begin{Definition}\label{def:s-r-fold-collinear}
Let $0\leq r\leq s-3$. We say that $X$ is an $(s-r)$-\emph{fold collinear configuration}
if there exists a line $L\subseteq \PP^2_k$ containing exactly $s-r$ points of $X$.
Writing
$Y:=X\cap L$, $Z:=X\setminus L$,
we have
$X=Y\cup Z$, $|Y|=s-r$, and $|Z|=r$.
We call $L$ a {\it distinguished line} of the configuration. 
\end{Definition}

The condition $r\leq s-3$ means that the distinguished line contains at least three
points. Thus the configuration has a genuine collinear component, while the residual
set $Z$ measures how far $X$ is from being collinear. By
Proposition~\ref{prop:projective-invariance}, relation type is unchanged under
projective linear changes of coordinates. Therefore, after a projective transformation,
we may assume that the distinguished line is $L=V(\ell)$, and in practice we shall often
take $\ell=z$. We write $a:=s-r=|Y|$. 

\medskip
The subset $Y\subseteq L$ defines a reduced effective divisor on $L=V(\ell)$. 
To say that the image of $F \in |I_Y|_a$ modulo $\ell$ cuts out the 
divisor $Y$ on $L$ means
that the zero scheme of $\overline F\in R/(\ell)$ on $L$ is exactly $Y$, and that
$I_Y=(\ell,F)$.

\begin{Proposition}\label{prop:sr-conic-characterization}
Let $X=Y\cup Z$ be an $(s-r)$-fold collinear configuration, where
$Y\subseteq L$ and $|Y|=s-r\geq 3$. Then $X$ is contained in a conic if and only if
the residual set $Z$ is collinear. 
\end{Proposition}

\begin{proof}
If $Z$ is collinear, say $Z\subseteq M$ for some line $M$, then $X=Y\cup Z\subseteq L\cup M$, and $L\cup M$ is a conic. 
Conversely, suppose that $X$ is contained in a conic $C$. Since $Y\subseteq L$ contains
at least three distinct points, the line $L$ meets $C$ in at least three distinct points.
Bézout's theorem now implies that
$L$ is a component of $C$; so $C=L\cup M$ for some line $M$. The fact that $Z\cap L=\emptyset$ forces $Z\subseteq M$. Thus $Z$ is
collinear.
\end{proof}

\begin{Corollary}\label{cor:sr-conic-linear-type}
Let $X=Y\cup Z$ be an $(s-r)$-fold collinear configuration. If the residual set $Z$ is
collinear, then $\rt(X)=1$. 
\end{Corollary}

\begin{proof}
By Proposition~\ref{prop:sr-conic-characterization}, the set $X$ is contained in a
conic. Hence Proposition~\ref{prop:points-on-conic} gives $\rt(X)=1$.
\end{proof}

\begin{Lemma}\label{lem:bdl-s-r}
Let $X=Y\cup Z$ be an $(s-r)$-fold collinear configuration, and set
$a:=|Y|=s-r$. Assume that there is a form $F\in (I_Z)_a$ whose image modulo $\ell$
cuts out the reduced divisor $Y$ on $L=V(\ell)$. Then
\[
I_X=\ell I_Z+(F).
\]
\end{Lemma}

\begin{proof}
Since $F$ cuts out $Y$ on $L$, we have $I_Y=(\ell,F)$. Hence
$I_X=I_Y\cap I_Z=(\ell,F)\cap I_Z$. 
The inclusion $\ell I_Z+(F)\subseteq (\ell,F)\cap I_Z$ is clear, because
$F\in I_Z$. Conversely, let $H\in(\ell,F)\cap I_Z$. Write $H=\ell A+BF$. Since
$H,F\in I_Z$, we get $\ell A\in I_Z$. As $Z\cap L=\emptyset$, the form $\ell$ is a
nonzerodivisor on $R/I_Z$, so $A\in I_Z$. Thus $H\in \ell I_Z+(F)$, proving the reverse
inclusion.
\end{proof}

We shall repeatedly use the following resolution attached to the decomposition above.

\begin{Proposition}\label{prop:bdl-resolution-s-r}
Let $X=Y\cup Z$ be as in Lemma~\ref{lem:bdl-s-r}, with distinguished line
$L=V(\ell)$ and $a=|Y|$. Assume that $Z\neq\emptyset$, and let $I_Z=(G_1,\ldots,G_q)$ be a minimal homogeneous generating set, with $\deg G_i=d_i$. Let
\[
0\lar \bigoplus_{j=1}^{q-1} R(-b_j)
\stackrel{\varphi}{\lar}
\bigoplus_{i=1}^{q} R(-d_i)
\lar I_Z
\lar 0
\]
be the Hilbert--Burch resolution of $I_Z$. Suppose that $F\in (I_Z)_a$ cuts out $Y$
modulo $\ell$, and write
\[
F=A_1G_1+\cdots+A_qG_q,
\]
where each $A_i$ is either zero or homogeneous of degree $a-d_i$. Set $A={}^{t}(A_1,\ldots,A_q)$. Then $I_X$ has a graded free resolution
\[
0\lar
\left(\bigoplus_{j=1}^{q-1}R(-b_j-1)\right)\oplus R(-a-1)
\stackrel{\Phi}{\lar}
\left(\bigoplus_{i=1}^{q}R(-d_i-1)\right)\oplus R(-a)
\lar I_X
\lar 0
\]
where
\[
\Phi=
\left(
\begin{array}{c|c}
\varphi & A\\ \hline
0\ \cdots\ 0 & -\ell
\end{array}
\right).
\]
\end{Proposition}

\begin{proof}
By Lemma~\ref{lem:bdl-s-r},
$I_X=\ell I_Z+(F)=(\ell G_1,\ldots,\ell G_q,F)$.
Let
\[
\psi:
\left(\bigoplus_{i=1}^{q}R(-d_i-1)\right)\oplus R(-a)
\lar I_X
\]
be the map sending the standard basis to 
$\ell G_1,\ldots,\ell G_q,F$.
The first $q-1$ columns of $\Phi$ are the syzygies of $G_1,\ldots,G_q$, now multiplied
by $\ell$, and the last column gives the relation
\[
A_1(\ell G_1)+\cdots+A_q(\ell G_q)-\ell F=0.
\]
Thus $\operatorname{im}\Phi\subseteq \ker\psi$.
Conversely, suppose that 
$\sum_{i=1}^{q}U_i\ell G_i+VF=0$. Reducing modulo $\ell$, we obtain 
$\overline V\,\overline F=0$ in $R/\ell$. 
Since $R/(\ell)$ is a domain and $\overline F\neq 0$, we have $V\in(\ell)$, say
$V=-\ell W$. Then
$\ell\left(\sum_{i=1}^{q}U_iG_i-WF\right)=0$.
Since $R$ is a domain,
$\sum_{i=1}^{q}(U_i-WA_i)G_i=0$. 
Therefore
$(U_1-WA_1,\ldots,U_q-WA_q)^t$
is a syzygy of $G_1,\ldots,G_q$, and hence lies in the image of $\varphi$. It follows that
$(U_1,\ldots,U_q,V)^t$ lies in the image of $\Phi$. Thus $\ker\psi=\operatorname{im}\Phi$, and the displayed
complex is exact.
\end{proof}

The resolution in Proposition~\ref{prop:bdl-resolution-s-r} is the standard basic double-link resolution of $I_X$. Since the Hilbert--Burch resolution of $I_Z$ is minimal, cancellations may only arise from scalar entries among the coefficients $A_i$. In particular, the displayed resolution is minimal whenever the generators $\ell G_1,\ldots,\ell G_q,F$ are minimal and every nonzero $A_i$ has positive degree.
We now dispose of those cases in which the residual block is too small to produce nonlinear relation types.

\begin{Proposition}\label{prop:r012}
Let $X\subseteq \PP^2_k$ be an $(s-r)$-fold collinear configuration. If $r=0,1,2$, then $\rt(X)=1$. 
\end{Proposition}
\begin{proof}
If $r=0$, then all points of $X$ are collinear, and the result follows from Example~\ref{ex:collinear-points}. If $r=1$ or $r=2$, then the residual set $Z$ is collinear, with the convention that one or two points are collinear. Hence
Corollary~\ref{cor:sr-conic-linear-type} gives $\rt(X)=1$.
\end{proof}

The same observation gives a useful class of configurations of linear type.

\begin{Proposition}\label{prop:two-collinear-blocks}
Let $X\subseteq \PP^2_k$ be a finite set of points contained in the union of two distinct lines. Then $\rt(X)=1$. In particular, if $X$ is the union of two reduced collinear blocks on two distinct lines, then $\rt(X)=1$.
\end{Proposition}

\begin{proof}
Since $X$ is contained in the union of two distinct lines, it is contained in a reducible conic. Therefore Proposition~\ref{prop:points-on-conic} gives $\rt(X)=1$.
\end{proof}
The following lemma isolates the determinant-and-linkage argument arising from the Jacobian dual identity and used below.
\begin{Lemma}\label{lem:jacobian-dual-linkage}
Let $S=R[T_1,T_2,T_3,T_4]$, where $R=k[x,y,z]$, and let $\fm=(x,y,z)$. Let
$\mathcal L=(L_1,L_2,L_3)\subseteq S$ be a complete intersection such that each $L_i$
has $T$-degree $1$ and $\mathcal L\subseteq \fm S$. Let $B$ be a $3\times 3$ matrix, 
with homogeneous entries in $S$ of $T$-degree $1$, for which
\[
[x\ y\ z]B=[L_1\ L_2\ L_3].
\]
Set $D:=\det B$ and $J:=(\mathcal L,D)$. If the image $\overline D$ of $D$ in
$S/\fm S$ is nonzero, then $\mathcal L:\fm S=J$. If, in addition,
$\mathcal L:J=\fm S$, then $S/J$ is Cohen--Macaulay.
\end{Lemma}
\begin{proof}
Multiplying $[x\ y\ z]B=[L_1\ L_2\ L_3]$ by $\operatorname{adj}(B)$ gives
$[x\ y\ z]D=[L_1\ L_2\ L_3]\operatorname{adj}(B)$. Hence
$xD,yD,zD\in\mathcal L$, and so $J\subseteq \mathcal L:\fm S$.
Set $A:=S/\mathcal L$ and $\fn:=\fm A$. Then
\[
(\mathcal L:\fm S)/\mathcal L=0:_A\fn=\operatorname{Hom}_A(A/\fn,A).
\]
Since $\mathcal L\subseteq \fm S$, we have $A/\fn\simeq S/\fm S$. By
\cite[Lemma~1.2.4]{BH},
\[
\operatorname{Hom}_A(A/\fn,A)\simeq \operatorname{Ext}^3_S(S/\fm S,S)(-3).
\]
The Koszul resolution associated with the $S$-regular sequence $x,y,z$ gives
$\operatorname{Ext}^3_S(S/\fm S,S)\simeq S/\fm S$, with no additional shift in $T$-degree.
Therefore $(\mathcal L:\fm S)/\mathcal L\simeq (S/\fm S)(-3)$.

\medskip
Thus $(\mathcal L:\fm S)/\mathcal L$ is a free cyclic $S/\fm S$-module generated in
$T$-degree $3$. Since $D$ has $T$-degree $3$ and $\overline D\neq 0$, the class of $D$
is a nonzero element in the first nonzero graded piece; so it generates this cyclic
module. Therefore $\mathcal L:\fm S=(\mathcal L,D)=J$.

\medskip
Finally, assume that $\mathcal L:J=\fm S$. Then $J$ and $\fm S$ are directly linked by
the complete intersection $\mathcal L$. Since
$S/\fm S\simeq k[T_1,T_2,T_3,T_4]$ is Cohen--Macaulay, linkage by the complete
intersection $\mathcal L$ yields that $S/J$ is Cohen--Macaulay; see
\cite[Theorem~21.23]{EisCA}.
\end{proof}
We now study $(s-r)$-fold collinear configurations whose residual set is itself almost
collinear. The next theorem determines exactly when such a configuration has relation
type one and when it has relation type three.

\begin{Theorem}\label{thm:rt-residual-r-1-fold}
Let $X=Y\cup Z$ be an $(s-r)$-fold collinear configuration with $r\geq 4$, where
$|Y|=a:=s-r$, $|Z|=r$, and $Y$ lies on the distinguished line $L$. Assume that $Z$ is
itself an $(r-1)$-fold collinear configuration, that is, $r-1$ points of $Z$ lie on a
line $M$ and one point of $Z$ lies off $M$. Assume also that $a\geq r-1$, and set
$P:=L\cap M$. Then the following hold:
\begin{enumerate}
\item[\rm (i)] if $a=r-1$ and $P\notin Y$, then $\rt(X)=1$;
\item[\rm (ii)] in all other cases, $\rt(X)=3$.
\end{enumerate}
\end{Theorem}

\begin{proof}
By Proposition~\ref{prop:projective-invariance}, we may assume that
$L=V(z)$, $M=V(y)$, and that the point of $Z$ off $M$ is $p=[0:1:1]$. Write the
$r-1$ points of $Z$ on $M$ as $[a_i:0:1]$, $i=1,\ldots,r-1$, with the $a_i$ pairwise
distinct. Thus $P=L\cap M=[1:0:0]$. For each $i$, set
\[
\lambda_i:=x+a_i y-a_i z,
\qquad
H:=\prod_{i=1}^{r-1}\lambda_i.
\]
Then
\[
I_Z=(G_1,G_2,G_3),\qquad
G_1=y\lambda_1,\quad G_2=y\lambda_2,\quad G_3=H,
\]
and a Hilbert--Burch matrix of $I_Z$ is
\[
\varphi=
\begin{pmatrix}
\lambda_2 & H/\lambda_1\\
-\lambda_1 & 0\\
0 & -y
\end{pmatrix}.
\]

Let $g\in k[x,y]_a$ be the reduced equation of $Y$ on $L=V(z)$. Modulo $z$, we have
$\overline G_1=y(x+a_1y)$, $\overline G_2=y(x+a_2y)$, and
$\overline G_3=\prod_{i=1}^{r-1}(x+a_i y)$. Since
$\overline G_2-\overline G_1=(a_2-a_1)y^2$, we get $y^2\in
(\overline G_1,\overline G_2,\overline G_3)$, and then also
$xy\in(\overline G_1,\overline G_2,\overline G_3)$. Moreover
$\overline G_3=x^{r-1}+yQ$ for some $Q\in k[x,y]_{r-2}$, so
$x^{r-1}$ belongs to the same ideal. Hence, because $a\geq r-1$, every form of degree
$a$ in $k[x,y]$ belongs to
$(\overline G_1,\overline G_2,\overline G_3)$. Thus we may write
\[
g=A_1\overline G_1+A_2\overline G_2+B\overline G_3,
\]
with $A_1,A_2\in k[x,y]_{a-2}$ and $B\in k[x,y]_{a-r+1}$. Choosing homogeneous lifts
to $R$, set $F:=A_1G_1+A_2G_2+BG_3\in (I_Z)_a$. Then $F$ cuts out $Y$ modulo $z$, and
Lemma~\ref{lem:bdl-s-r} gives
\[
I_X=zI_Z+(F)=(zG_1,zG_2,zG_3,F).
\]

Assume first that $a=r-1$ and $P\notin Y$. Then $B$ is a scalar. Since
$\overline G_1(P)=\overline G_2(P)=0$ and $\overline G_3(P)=1$, we have $B=g(P)\neq 0$.
Thus $zF=A_1zG_1+A_2zG_2+BzG_3$, and hence $zG_3\in(zG_1,zG_2,F)$. Therefore
$I_X=(zG_1,zG_2,F)$ is generated by at most three elements. Since $\hht I_X=2$ and
$X$ is reduced, Remark~\ref{rem:ci-linear-type} gives $\rt(X)=1$.

\medskip
Now consider all remaining cases. Then $B\in\fm=(x,y,z)$: indeed, if $a>r-1$, then
$\deg B>0$, while if $a=r-1$ and $P\in Y$, then $B=g(P)=0$. Also
$A_1,A_2\in\fm$, because $\deg A_1=\deg A_2=a-2\geq r-3\geq 1$. Hence
$F\in\fm I_Z$.

\medskip
We claim that $zG_1,zG_2,zG_3,F$ is a minimal generating set of $I_X$. Clearly
$F\notin zI_Z$, since its image modulo $z$ is $g\neq 0$. If some $zG_i$ were redundant,
then
\[
zG_i=\sum_{j\neq i}u_jzG_j+vF.
\]
Reducing modulo $z$ gives $\overline v\,g=0$ in $k[x,y]$, hence $v=zw$. Cancelling
$z$, we get $G_i=\sum_{j\neq i}u_jG_j+wF$. Passing to $I_Z/\fm I_Z$ and using
$F\in\fm I_Z$ contradicts the minimality of $G_1,G_2,G_3$. Thus $\mu(I_X)=4$, so
$I_X$ is not of linear type by Remark~\ref{rem:not-linear-type-many-generators}.

\medskip
By Proposition~\ref{prop:bdl-resolution-s-r}, a Hilbert--Burch matrix of $I_X$ is
\[
\Phi=
\left(
\begin{array}{cc|c}
\lambda_2 & H/\lambda_1 & A_1\\
-\lambda_1 & 0 & A_2\\
0 & -y & B\\ \hline
0 & 0 & -z
\end{array}
\right)
\]
with respect to the ordered generators $zG_1,zG_2,zG_3,F$. Let
$S=R[T_1,T_2,T_3,T_4]$, and let $\mathcal L$ be the defining ideal of
$\operatorname{Sym}_R(I_X)$. Then $\mathcal L=(\mathcal L_1,\mathcal L_2,\mathcal L_3)$,
where
\[
\mathcal L_1=\lambda_2T_1-\lambda_1T_2,\quad
\mathcal L_2=(H/\lambda_1)T_1-yT_3,\quad
\mathcal L_3=A_1T_1+A_2T_2+BT_3-zT_4.
\]
Since $I_1(\Phi)=\fm$, the Jacobian dual construction applies. Let $B(\Phi)$ be a
Jacobian dual matrix satisfying
\[
[x\ y\ z]B(\Phi)=[\mathcal L_1\ \mathcal L_2\ \mathcal L_3]
\]
and set $D:=\det B(\Phi)$ and $J:=(\mathcal L,D)$.

\medskip
We use Lemma~\ref{lem:jacobian-dual-linkage}. We first show that $\mathcal L$ is a complete intersection. Since
$\mathcal L$ defines $\operatorname{Sym}_R(I_X)$, the Huneke--Rossi dimension formula
\cite{HunekeRossi} gives
\[
\dim \operatorname{Sym}_R(I_X)
=\sup_{\fp\in\operatorname{Spec}(R)}
\{\dim R/\fp+\mu((I_X)_\fp)\}.
\]
This supremum is $4$: outside $V(I_X)$ the localized ideal is principal, on the punctured
spectrum over $V(I_X)$ it is generated by two elements, and at $\fm$ it has four minimal
generators. Hence $\dim S/\mathcal L=4$. Since $\dim S=7$, we have
$\hht(\mathcal L)=3$, and therefore the three generators of $\mathcal L$ form a complete
intersection.

\medskip
Next, $\overline D\neq 0$ in $S/\fm S$. To see this, write
$H/\lambda_1=xU_1+yU_2+zU_3$. Since $r\geq 4$, the $U_i$ have positive degree, and
hence vanish modulo $\fm$. Also $A_1,A_2,B\in\fm$. Therefore, modulo $\fm S$, the only
term involving $T_4$ in $\overline D$ comes from the entry $-T_4$, and its coefficient is
$(T_1-T_2)T_3$. Hence
\[
\overline D=(T_1-T_2)T_3T_4+\text{terms not involving }T_4
\]
so $\overline D\neq 0$.

\medskip
Moreover, $\mathcal L:J=\fm S$. To wit: the inclusion $\fm S\subseteq\mathcal L:J$ follows from
$xD,yD,zD\in\mathcal L$. Conversely, if $H_0D\in\mathcal L$, then reducing modulo
$\fm S$ gives $\overline{H_0}\,\overline D=0$ in the domain
$k[T_1,T_2,T_3,T_4]$. Since $\overline D\neq 0$, we get $H_0\in\fm S$. Lemma~\ref{lem:jacobian-dual-linkage} now gives
\[
\mathcal L:\fm S=(\mathcal L,D)=J
\]
and shows that $S/J$ is Cohen--Macaulay. Since $\overline D\neq 0$, we have $J\not\subseteq \fm S$. As $S/J$ is Cohen--Macaulay of dimension $4$, its associated primes are minimal of height $3$. Thus no associated
prime of $S/J$ contains $\fm S$; otherwise it would equal $\fm S$, contradicting $J\not\subseteq \fm S$. Hence $S/J$ has no nonzero $\fm S$-torsion. 

\medskip
By the saturation description  for ideals of points, we have $\Q=\mathcal L:\fm^\infty$. Thus $J=\mathcal L:\fm S\subseteq \Q$, and $\Q/J$ is $\fm S$-torsion. Since
$\Q/J\subseteq S/J$ and $S/J$ has no nonzero $\fm S$-torsion, we get $\Q=J$.

\medskip
Thus the defining ideal of the Rees algebra is generated by the three linear equations
of $\mathcal L$ and the cubic equation $D$. Therefore $\rt(X)\leq 3$. Since $I_X$ is not
of linear type and relation type $2$ does not occur for finite sets of points in
$\PP^2_k$, we conclude that $\rt(X)=3$.
\end{proof}
\subsection{Triangular residual blocks in generic position}
\label{subsec:triangular-residual-blocks}
We recall the convention of Geramita--Maroscia \cite[Definition~1.5]{GM}. A finite set
$Z\subseteq \PP^2_k$ of $r$ points is said to be in \emph{generic position} if its Hilbert
function is the generic one, namely
\[
\HF_{R/I_Z}(t)=\min\left\{r,\binom{t+2}{2}\right\}
\quad\text{for all }t\geq 0.
\]
Equivalently, for every $t\geq 0$, the evaluation map
\[
R_t\longrightarrow k^r,\qquad F\longmapsto (F(p))_{p\in Z},
\]
has maximal rank.

Assume that $|Z|=\binom{d+1}{2}$ and that $Z$ is in generic position. Then $(I_Z)_{d-1}=0$ and, by Example~\ref{ex:triangular-number-general-points}, if
$d\geq 3$, then $\rt(Z)=3$.

We now introduce the residual blocks that will be used in the next results.
\begin{Definition}\label{def:triangular-residual-block}
Let $X=Y\cup Z$ be an $(s-r)$-fold collinear configuration. We say that $X$ has a
\emph{$d$-triangular residual block in generic position} if
\[
|Z|=\binom{d+1}{2}
\qquad\text{and}\qquad
(I_Z)_{d-1}=0.
\]
Equivalently, $Z$ consists of $\binom{d+1}{2}$ points in generic position.
\end{Definition}
The next lemma records the basic double-link description in the case where the residual
block has triangular cardinality and generic Hilbert function.

\begin{Lemma}\label{lem:triangular-block}
Let $X=Y\cup Z$ be an $(s-r)$-fold collinear configuration with distinguished line $L$,
and set $a:=|Y|=s-r$. Assume that $X$ has a $d$-triangular residual block in generic
position and that $a\geq d+1$. Then
\[
I_X=(h_0,\ldots,h_d,F),
\]
where $\deg h_i=d+1$ and $\deg F=a$. Moreover, $I_X$ has a minimal free resolution
\[
0\lar R(-d-2)^d\oplus R(-a-1)
\stackrel{\Phi}{\lar}
R(-d-1)^{d+1}\oplus R(-a)
\lar I_X\lar 0,
\]
where
\[
\Phi=
\left(
\begin{array}{c|c}
\varphi & A\\ \hline
0\ \cdots\ 0 & -z
\end{array}
\right),
\]
$\varphi$ is a linear Hilbert--Burch matrix of $I_Z$, and
$A={}^{t}(A_0,\ldots,A_d)$ with $A_i\in R_{a-d}$.
\end{Lemma}

\begin{proof}
After a projective change of coordinates, assume $L=V(z)$. By
Example~\ref{ex:triangular-number-general-points}, the ideal $I_Z$ is generated by
$d+1$ forms $G_0,\ldots,G_d$ of degree $d$ and has a linear Hilbert--Burch resolution
\[
0\lar R(-d-1)^d\stackrel{\varphi}{\lar}R(-d)^{d+1}\lar I_Z\lar 0.
\]

The images $\overline G_0,\ldots,\overline G_d$ form a basis of $k[x,y]_d$. Indeed, if
$G\in(I_Z)_d$ and $\overline G=0$, then $G=zH$ with $H\in R_{d-1}$. Since
$Z\cap V(z)=\emptyset$, the form $z$ is a nonzerodivisor on $R/I_Z$, so
$H\in(I_Z)_{d-1}=0$. Thus the reduction map $(I_Z)_d\lar k[x,y]_d$ is injective, hence
an isomorphism.

Let $g\in k[x,y]_a$ be the reduced equation of $Y$ on $L$. Since
$(\overline G_0,\ldots,\overline G_d)=(x,y)^d$ and $a\geq d+1$, we may write
\[
g=A_0\overline G_0+\cdots+A_d\overline G_d,
\qquad A_i\in k[x,y]_{a-d}.
\]
Lifting the $A_i$ to $R$, set $F=A_0G_0+\cdots+A_dG_d\in(I_Z)_a$. Then $F$ cuts out
$Y$ modulo $z$, and Lemma~\ref{lem:bdl-s-r} gives
\[
I_X=zI_Z+(F)=(zG_0,\ldots,zG_d,F).
\]
Set $h_i:=zG_i$. Since $\overline F=g\neq 0$, we have $F\notin zI_Z$. Also, because
$a\geq d+1$, all nonzero $A_i$ have positive degree, and hence $F\in\fm I_Z$.

It remains only to check minimality. If some $h_i=zG_i$ were redundant, then
\[
zG_i=\sum_{j\neq i}u_jzG_j+vF.
\]
Reducing modulo $z$ gives $\overline v\,g=0$ in $k[x,y]$, hence $v=zw$. Cancelling $z$,
we obtain
\[
G_i=\sum_{j\neq i}u_jG_j+wF.
\]
Passing to $I_Z/\fm I_Z$ and using $F\in\fm I_Z$ contradicts the minimality of
$G_0,\ldots,G_d$. Therefore $h_0,\ldots,h_d,F$ are minimal generators. The displayed
minimal resolution follows from Proposition~\ref{prop:bdl-resolution-s-r}.
\end{proof}
\begin{Remark}\label{rem:triangular-block-small-a}
The assumption $a\geq d+1$ in Lemma~\ref{lem:triangular-block} is exactly the condition
which prevents cancellations in the basic double-link resolution. Indeed, since $Z$ is
$d$-triangular in generic position, Example~\ref{ex:triangular-number-general-points}
gives $\operatorname{indeg}(I_Z)=d$. Hence, if $a<d$, then $(I_Z)_a=0$, so no form of
degree $a$ in $I_Z$ can cut out $Y$ on $L$.

If $a=d$, then the reduction map $(I_Z)_d\lar k[x,y]_d$ is an isomorphism, so such a form
$F\in (I_Z)_d$ exists. However, $F$ is a constant linear combination of the degree-$d$
minimal generators of $I_Z$. After changing the basis of $(I_Z)_d$, one may assume
$F=G_0$. Then
\[
I_X=zI_Z+(F)=(F,zG_1,\ldots,zG_d),
\]
because $zG_0=zF$ is redundant. Thus a cancellation occurs in the basic double-link
resolution. For this reason, Lemma~\ref{lem:triangular-block} is stated under the
hypothesis $a\geq d+1$.
\end{Remark}
We now show that, once the collinear block is sufficiently large, triangular residual
blocks in generic position force relation type three.

\begin{Theorem}\label{thm:rt-triangular-block}
Let $X=Y\cup Z\subseteq \PP^2_k$ be an $(s-r)$-fold collinear configuration, and set
$a:=|Y|=s-r$. Assume that $X$ has a $d$-triangular residual block in generic position,
with $d\geq 2$. If $a\geq d+1$, then $\rt(X)=3$.
\end{Theorem}

\begin{proof}
By Lemma~\ref{lem:triangular-block}, we may write
\[
I_X=(zG_0,\ldots,zG_d,F),
\]
where $G_0,\ldots,G_d$ minimally generate $I_Z$ and
$F=A_0G_0+\cdots+A_dG_d$. A Hilbert--Burch matrix of $I_X$ is
\[
\Phi=
\left(
\begin{array}{c|c}
\varphi & A\\ \hline
0\ \cdots\ 0 & -z
\end{array}
\right),
\]
where $\varphi$ is linear and $A={}^{t}(A_0,\ldots,A_d)$.

If $a=d+1$, then all entries of $A$ are linear, so $\Phi$ is linear. Since
$\mu(I_X)=d+2\geq 4$, Proposition~\ref{prop:points-plane-gap} gives $\rt(X)=3$.

Assume now that $a>d+1$. Then all entries of $A$ have degree at least $2$. Since the reduction map $(I_Z)_d\lar k[x,y]_d$ is an isomorphism, we may replace the
minimal generators $G_0,\ldots,G_d$ of $I_Z$ by a suitable $k$-linear combination and
assume that $\overline G_i=x^{d-i}y^i$ in $R/(z)$ for $i=0,\ldots,d$. With this choice, the reduction modulo $z$ of a Hilbert--Burch matrix of $I_Z$ is the
Hilbert--Burch matrix of $(x,y)^d$, namely
\[
\overline\varphi=
\begin{pmatrix}
y & 0 & \cdots & 0\\
-x & y & \ddots & \vdots\\
0 & -x & \ddots & 0\\
\vdots & \ddots & \ddots & y\\
0 & \cdots & 0 & -x
\end{pmatrix}.
\]
Let $S=R[T_0,\ldots,T_d,T]$, and let $B(\Phi)$ be a Jacobian dual matrix. Modulo
$\fm S$, the last column of $B(\Phi)$ is ${}^{t}(0,0,-T)$, while the first two rows
contain the catalecticant matrix
\[
N=
\begin{pmatrix}
T_1 & T_2 & \cdots & T_d\\
T_0 & T_1 & \cdots & T_{d-1}
\end{pmatrix}.
\]
Hence $I_3(\overline{B(\Phi)})$ contains $T\,I_2(N)$. Since $I_2(N)$ defines the
rational normal curve in $\PP^d$, we have $\hht I_2(N)=d-1$.

Let $\mathfrak p$ be a prime containing $I_3(\overline{B(\Phi)})$. If
$T\notin\mathfrak p$, then $I_2(N)\subseteq\mathfrak p$, and hence
$\hht \mathfrak p\geq d-1$. If $T\in\mathfrak p$, then
\[
(I_3(\overline{B(\varphi)}),T)\subseteq\mathfrak p.
\]
For $d=2$, this already gives height at least $1=d-1$. For $d\geq 3$, applying the
Morey--Ulrich theorem to the linearly presented ideal $I_Z$ gives
\[
\dim k[T_0,\ldots,T_d]/I_3(\overline{B(\varphi)})\leq 3,
\]
and hence $\hht I_3(\overline{B(\varphi)})\geq d-2$. Therefore
\[
\hht (I_3(\overline{B(\varphi)}),T)\geq d-1.
\]
Thus
\[
\hht\frac{I_3(B(\Phi))+\fm S}{\fm S}\geq d-1.
\]

Since $\mu(I_X)=d+2$, this is precisely the height condition in Morey's criterion
\cite[Proposition~3.1]{Morey}. Hence the defining ideal of the Rees algebra of $I_X$ is generated by the
symmetric equations and the $3$-minors of $B(\Phi)$. Therefore $\rt(X)\leq 3$.

On the other hand, $\mu(I_X)=d+2\geq 4$, so $I_X$ is not of linear type by
Remark~\ref{rem:not-linear-type-many-generators}. Since relation type $2$ does not
occur for finite sets of points in $\PP^2_k$, Proposition~\ref{prop:points-plane-gap}
gives $\rt(X)=3$.
\end{proof}

\begin{Corollary}\label{cor:rt-s-3-fold}
Let $X\subseteq \PP^2_k$ be an $(s-3)$-fold collinear configuration of $s\geq 6$
distinct points. Assume that the remaining three points are not collinear. Then
$\rt(X)=3$.
\end{Corollary}
\begin{proof}
This is Theorem~\ref{thm:rt-triangular-block} for $d=2$. Indeed,
$|Z|=3=\binom{3}{2}$, the condition that the three residual points are not collinear is
$(I_Z)_1=0$, and $|Y|=s-3\geq 3=d+1$.
\end{proof}
We now specialize the preceding results to the first nontrivial residual size, namely $r=4$. In this case the relation type is completely determined by the
collinearity pattern of the four residual points.

\begin{Proposition}\label{prop:r4-classification}
Let $X=Y\cup Z$ be an $(s-4)$-fold collinear configuration of $s\geq 7$ distinct points,
where $Y$ lies on the distinguished line $L$ and $|Z|=4$. Then the following hold.

\begin{enumerate}
\item[\rm (i)] One has $\rt(X)=1$ if and only if one of the following holds:
\begin{enumerate}
\item[\rm (a)] the four points of $Z$ are collinear;
\item[\rm (b)] no three points of $Z$ are collinear;
\item[\rm (c)] exactly three points of $Z$ are collinear on a line $M$, and, writing
$P:=L\cap M$, one has $s=7$ and $P\notin Y$.
\end{enumerate}

\item[\rm (ii)] One has $\rt(X)=3$ if and only if exactly three points of $Z$ are
collinear on a line $M$, and, writing $P:=L\cap M$, either $s>7$, or $s=7$ and
$P\in Y$.
\end{enumerate}
\end{Proposition}

\begin{proof}
There are three possibilities for the residual block $Z$: all four points are collinear,
no three are collinear, or exactly three are collinear.

If the four points of $Z$ are collinear, then $X$ is contained in the union of two
distinct lines, and hence $\rt(X)=1$ by Proposition~\ref{prop:two-collinear-blocks}.

Assume next that no three points of $Z$ are collinear. After a projective change of
coordinates, let $L=V(z)$. Set $a:=|Y|=s-4$, and let $g\in k[x,y]_a$ be the reduced
equation of $Y$ on $L$. Since no three points of $Z$ are collinear, $I_Z$ is a complete
intersection of two quadrics, say $I_Z=(q_1,q_2)$. The images
$\overline q_1,\overline q_2\in k[x,y]$ have no common zero on $L$, because
$Z\cap L=\emptyset$. Thus $k[x,y]/(\overline q_1,\overline q_2)$ has Hilbert function
$1,2,1$, and so $(\overline q_1,\overline q_2)_j=k[x,y]_j$ for every $j\geq 3$.
Since $a=s-4\geq 3$, there exists $F\in (I_Z)_a$ whose image modulo $z$ is $g$. Hence
Lemma~\ref{lem:bdl-s-r} gives
\[
I_X=zI_Z+(F)=(zq_1,zq_2,F).
\]
Thus $I_X$ is generated by at most three elements. Since $\hht I_X=2$ and $X$ is
reduced, Remark~\ref{rem:ci-linear-type} gives $\rt(X)=1$.

It remains to treat the case where exactly three points of $Z$ are collinear. Let $M$
be the line containing them and set $P:=L\cap M$. Then $Z$ is an $(r-1)$-fold collinear
configuration with $r=4$. Here $a=s-4\geq 3=r-1$, so
Theorem~\ref{thm:rt-residual-r-1-fold} applies and gives
\[
\rt(X)=1
\quad\Longleftrightarrow\quad
a=r-1 \text{ and } P\notin Y.
\]
Since $a=r-1$ is equivalent to $s=7$, we get $\rt(X)=1$ exactly when
$s=7$ and $P\notin Y$; in all other cases of this type, $\rt(X)=3$. This proves both
assertions.
\end{proof}
\subsection{$(s-5)$-fold collinear configurations}

We now consider the case $r=5$. Let $X=Y\cup Z$ be an $(s-5)$-fold collinear
configuration, where $|Y|=s-5$, $|Z|=5$, and $Y$ lies on the distinguished line $L$.
Thus $s\geq 8$. The residual block $Z$ has one of the following incidence types:
\begin{enumerate}
\item[\rm (i)] the five points of $Z$ are collinear;
\item[\rm (ii)] exactly four points of $Z$ are collinear and the fifth point lies off
that line;
\item[\rm (iii)] $Z$ has two distinct three-point collinear subsets;
\item[\rm (iv)] $Z$ has exactly one three-point collinear subset;
\item[\rm (v)] no three points of $Z$ are collinear.
\end{enumerate}
In cases {\rm (iii)--(v)}, let $C_Z$ denote the unique conic through $Z$; it may be
reducible in cases {\rm (iii)} and {\rm (iv)}.

\begin{Proposition}\label{prop:r5-classification}
Let $X=Y\cup Z$ be an $(s-5)$-fold collinear configuration of $s\geq 8$ distinct
points. Assume, in cases {\rm (iii)--(v)}, that the conic $C_Z$ meets $L$ in two
distinct points. Then $\rt(X)=1$ precisely in the following cases:
\begin{enumerate}
\item[\rm (a)] the five points of $Z$ are collinear;
\item[\rm (b)] exactly four points of $Z$ are collinear on a line $M$, the fifth point
lies off $M$, and, writing $P:=L\cap M$, one has $s=9$ and $P\notin Y$;
\item[\rm (c)] $s=8$, $Z$ is of type {\rm (iii)}, {\rm (iv)}, or {\rm (v)}, and
$C_Z$ contains at most one point of $Y$.
\end{enumerate}
In all remaining cases, $\rt(X)=3$.
\end{Proposition}

\begin{proof}
If the five points of $Z$ are collinear, then $X$ is contained in the union of two
lines, so $\rt(X)=1$ by Proposition~\ref{prop:two-collinear-blocks}.

Assume next that exactly four points of $Z$ are collinear on a line $M$, and let the
fifth point lie off $M$. Put $P:=L\cap M$ and $a:=|Y|=s-5$. If $s\geq 9$, then
$a\geq 4$, and Theorem~\ref{thm:rt-residual-r-1-fold} gives $\rt(X)=1$ exactly when
$a=4$ and $P\notin Y$, equivalently, when $s=9$ and $P\notin Y$; all other cases have
relation type $3$. If $s=8$, then $a=3$. When $P\in Y$, the line $M$ contains five
points of $X$, and the remaining three points are not collinear, so
Corollary~\ref{cor:rt-s-3-fold} gives $\rt(X)=3$. When $P\notin Y$, the line $M$
contains exactly four points of $X$, and the residual four-point block has exactly
three collinear points; hence Proposition~\ref{prop:r4-classification} gives
$\rt(X)=3$.

It remains to treat cases {\rm (iii)--(v)}. Set $a:=s-5$. Since no four points of
$Z$ are collinear, $\HF_{R/I_Z}=1,3,5,5,\ldots$, and hence
\[
I_Z=(q,c_1,c_2),\qquad \deg q=2,\quad \deg c_1=\deg c_2=3.
\]
By the transverse-conic hypothesis, after a projective change of coordinates we may
assume $L=V(z)$ and
\[
\overline q=xy,\qquad \overline c_1=x^3,\qquad \overline c_2=y^3
\quad\text{in }R/(z)\simeq k[x,y].
\]
Let $g\in k[x,y]_a$ be the reduced equation of $Y$ on $L$. Since
$(xy,x^3,y^3)_j=k[x,y]_j$ for all $j\geq 3$, we may write
\[
g=Axy+Bx^3+Cy^3,
\qquad
A\in k[x,y]_{a-2},\quad B,C\in k[x,y]_{a-3}.
\]
Set $F:=Aq+Bc_1+Cc_2$. Then $F$ cuts out $Y$ modulo $z$, and
Lemma~\ref{lem:bdl-s-r} gives
\[
I_X=zI_Z+(F)=(zq,zc_1,zc_2,F).
\]

Suppose first that $a=3$ and $(B,C)\neq (0,0)$. This is equivalent to saying that
$C_Z$ contains at most one point of $Y$. If, for instance, $B\neq 0$, then
$zF=Azq+Bzc_1+Czc_2$, so $zc_1$ is redundant; the case $C\neq 0$ is identical.
Thus $\mu(I_X)\leq 3$, and Remark~\ref{rem:ci-linear-type} gives $\rt(X)=1$.

Now assume that $a>3$, or that $a=3$ and $B=C=0$. In the latter case, $C_Z$ contains
two points of $Y$. In both cases $F\in \fm I_Z$. The four generators
$zq,zc_1,zc_2,F$ are minimal: indeed, $F\notin zI_Z$, and if some $zG_i$, with
$G_i\in\{q,c_1,c_2\}$, were redundant, reducing a relation modulo $z$ would force the
coefficient of $F$ to be divisible by $z$; after cancellation this contradicts the
minimality of $q,c_1,c_2$ in $I_Z/\fm I_Z$. Hence $\mu(I_X)=4$, so $I_X$ is not of
linear type.

It remains to show that $\rt(X)\leq 3$. By Proposition~\ref{prop:bdl-resolution-s-r},
$I_X$ has a Hilbert--Burch matrix
\[
\Phi=
\left(
\begin{array}{cc|c}
 & & A\\
 &\varphi& B\\
 & & C\\ \hline
0&0&-z
\end{array}
\right),
\]
where $\varphi$ is a Hilbert--Burch matrix of $I_Z$. With respect to the generators
$zq,zc_1,zc_2,F$, let $\mathcal L\subseteq S=R[T_0,T_1,T_2,T_3]$ be the symmetric
ideal, and let $B(\Phi)$ be a Jacobian dual matrix. Modulo $\fm S$, one may choose
$B(\Phi)$ so that
\[
\overline{B(\Phi)}=
\begin{pmatrix}
0&-T_2&H_1\\
-T_1&0&H_2\\
0&0&H_3-T_3
\end{pmatrix},
\qquad H_i\in k[T_0,T_1,T_2]_1.
\]
Thus, for $D:=\det B(\Phi)$, one has
$\overline D=T_1T_2(T_3-H_3)\neq 0$. Setting $J:=(\mathcal L,D)$, the same
Jacobian-dual linkage argument used in Theorem~\ref{thm:rt-residual-r-1-fold} gives
\[
\mathcal L:_S\fm S=J.
\]
By the saturation description of the Rees ideal of a point ideal,
$\Q=\mathcal L:\fm^\infty$, where $\Q$ is the defining ideal of $\Rees_R(I_X)$.
Since $S/J$ has no nonzero $\fm S$-torsion, it follows that $\Q=J$. Hence the Rees
ideal is generated in $T$-degrees at most $3$, and so $\rt(X)\leq 3$.

Since $I_X$ is not of linear type in the remaining cases and relation type $2$ does
not occur for finite sets of points in $\PP^2_k$, Proposition~\ref{prop:points-plane-gap}
gives $\rt(X)=3$.
\end{proof}
\begin{Remark}\label{rem:r5-tangent-conic}
The transverse hypothesis in Proposition~\ref{prop:r5-classification} is used only to
put the restriction of the quadratic generator in the normal form $\overline q=xy$.
If the conic $C_Z$ meets $L$ non-reducedly, the same proof applies with the normal form
$\overline q=x^2$. Indeed, after choosing coordinates with $L=V(z)$, one may choose
generators
\[
I_Z=(q,c_1,c_2),\qquad \deg q=2,\quad \deg c_1=\deg c_2=3,
\]
so that in $R/(z)\simeq k[x,y]$, we have $\overline q=x^2$, $\overline c_1=xy^2$ and  $\overline c_2=y^3$. 
Since
$(x^2,xy^2,y^3)_j=k[x,y]_j$ for all $j\geq 3$,
the reduced equation $g\in k[x,y]_a$ of $Y$ on $L$ can be written as
\[
g=Ax^2+Bxy^2+Cy^3,
\]
and the basic double-link construction gives
\[
I_X=zI_Z+(F),\qquad F=Aq+Bc_1+Cc_2.
\]
The rest of the proof is unchanged. In particular, when $a=3$, the case
$(B,C)\neq(0,0)$ gives a redundant generator and hence $\rt(X)=1$, while the remaining
cases have $\mu(I_X)=4$ and the same Jacobian-dual linkage argument gives
$\rt(X)=3$.
\end{Remark}

\subsection{$(s-6)$-fold collinear configurations}
\label{subsec:s-6-fold}
Set $a:=|Y|=s-6$, so $a\geq 3$. A set of six points in $\PP^2_k$ has several incidence
types; classically one may distinguish eleven such types; see \cite[Figure~3.3]{AZ}.
For relation type, these cases are naturally grouped as follows:
\begin{enumerate}
\item[\rm (i)] $Z$ is collinear;
\item[\rm (ii)] exactly five points of $Z$ are collinear;
\item[\rm (iii)] exactly four points of $Z$ are collinear;
\item[\rm (iv)] $Z$ lies on a conic and no four points of $Z$ are collinear;
\item[\rm (v)] $Z$ lies on no conic.
\end{enumerate}
The separation of case \textup{(iii)} is necessary: if exactly four points of $Z$ are
collinear, then $Z$ lies on a reducible conic, but its defining ideal need not be a
complete intersection of type $(2,3)$. Thus this case is not covered by the conic
complete-intersection argument used in case \textup{(iv)}.

We first record a useful four-generated criterion.

\begin{Lemma}\label{lem:cubic-four-generated-rt}
Let $W\subseteq \PP^2_k$ be a finite reduced set of points with ideal $I=I_W$. Assume
that $(I)_2=0$, $\dim_k(I)_3=1$, and $\mu(I)=4$. Then $\rt(W)=3$.
\end{Lemma}

\begin{proof}
Let $f$ generate $(I)_3$, and choose minimal generators
$I=(f,f_1,f_2,f_3)$. Since $I$ is perfect of height two, it has a Hilbert--Burch
matrix $\Phi$ of size $4\times 3$. The maximal minor of the submatrix obtained by
deleting the first row is, up to sign, $f$. Hence this $3\times 3$ block has nonzero
determinant of degree $3$. Since all entries of a minimal Hilbert--Burch matrix have
positive degree, this determinant comes from the linear part of the block.

Let $S=R[T_0,T_1,T_2,T_3]$, let $\mathcal L$ be the symmetric ideal, and let
$B(\Phi)$ be a Jacobian dual matrix. The preceding nonzero linear block gives a
nonzero cubic minor of $\overline{B(\Phi)}$ modulo $\fm=(x,y,z)$. Thus
\[
\hht\frac{I_3(B(\Phi))+\fm S}{\fm S}\geq 1.
\]
Since $\mu(I)=4$ and $I$ satisfies $G_3$, Morey's criterion
\cite[Proposition~3.1]{Morey} gives
$\Q=(\mathcal L,I_3(B(\Phi)))$, where $\Q$ is the defining ideal of the Rees algebra. Hence $\rt(W)\leq 3$.
As $\mu(I)=4$, the ideal is not of linear type by
Remark~\ref{rem:not-linear-type-many-generators}; and relation type $2$ does not occur
for finite sets of points in $\PP^2_k$. Therefore $\rt(W)=3$.
\end{proof}

\begin{Proposition}\label{prop:r6-collinear-five-plus-one}
Let $X=Y\cup Z$ be an $(s-6)$-fold collinear configuration.
\begin{enumerate}
\item[\rm (i)] If the six points of $Z$ are collinear, then $\rt(X)=1$.
\item[\rm (ii)] Assume that exactly five points of $Z$ are collinear on a line $M$, and
the sixth point lies off $M$. Set $P:=L\cap M$. Then $\rt(X)=1$ if and only if
$s=11$ and $P\notin Y$. In all remaining cases, $\rt(X)=3$.
\end{enumerate}
\end{Proposition}

\begin{proof}
Part {\rm (i)} follows from Proposition~\ref{prop:two-collinear-blocks}. For {\rm (ii)},
set $a=s-6$. If $a\geq 5$, then Theorem~\ref{thm:rt-residual-r-1-fold}, applied with
$r=6$, gives $\rt(X)=1$ exactly when $a=5$ and $P\notin Y$, equivalently when
$s=11$ and $P\notin Y$; all other cases with $a\geq 5$ have relation type $3$.

It remains to treat $a=3,4$. If $a=3$ and $P\in Y$, then $M$ contains six points of
$X$, and the remaining three points are not collinear; hence
Corollary~\ref{cor:rt-s-3-fold} gives $\rt(X)=3$. If $a=3$ and $P\notin Y$, then
$M$ contains five points of $X$, and the residual four-point block consists of the
three points of $Y$ on $L$ together with the point of $Z$ off $M$. This block has
exactly three collinear points, and since $s=9>7$,
Proposition~\ref{prop:r4-classification} gives $\rt(X)=3$.

If $a=4$, the same reduction gives an $(s-4)$-fold or an $(s-5)$-fold collinear
configuration, according as $P\in Y$ or $P\notin Y$. In the first case
Proposition~\ref{prop:r4-classification} gives $\rt(X)=3$, and in the second case
Proposition~\ref{prop:r5-classification} gives $\rt(X)=3$.
\end{proof}

\begin{Proposition}\label{prop:r6-four-collinear}
Assume that exactly four points of $Z$ are collinear on a line $M$. Let $N$ be the line
through the remaining two points of $Z$, and set $P:=L\cap M$.
\begin{enumerate}
\item[\rm (i)] If $s=9$ and $P\in Y$, set
\[
W:=(Y\setminus\{P\})\cup (Z\setminus M).
\]
Then $\rt(X)=1$ if no three points of $W$ are collinear, and $\rt(X)=3$ if exactly
three points of $W$ are collinear.
\item[\rm (ii)] In all other cases, $X$ has no quadric equation and has a unique cubic
equation, namely the reducible cubic $LMN$. Consequently $\mu(I_X)\leq 4$, and
\[
\rt(X)=1 \Longleftrightarrow \mu(I_X)\leq 3,
\qquad
\rt(X)=3 \Longleftrightarrow \mu(I_X)=4.
\]
\end{enumerate}
\end{Proposition}

\begin{proof}
Assume first that $s=9$ and $P\in Y$. Then $M$ contains five points of $X$, so $X$ is
an $(s-4)$-fold collinear configuration with distinguished line $M$ and residual block
$W$. Since $|W|=4$, Proposition~\ref{prop:r4-classification} gives the assertion.

Now assume that we are not in this exceptional case. There is no quadric through $X$:
indeed, a quadric through the $a\geq 3$ points of $Y$ on $L$ either contains $L$ or,
when $a=3$, would still have to contain the line $M$ because it contains four points
of $Z$ on $M$; in either case the remaining component cannot contain all of $Z$.

Let $C$ be a cubic through $X$. Since $C$ contains four points of $Z$ on $M$, it
contains $M$. If $a\geq 4$, then $C$ also contains $L$, and the remaining linear factor
must be $N$. If $a=3$ and $P\notin Y$, then after removing the factor $M$, the residual
conic contains the three points of $Y$ on $L$, hence contains $L$, and again the
remaining factor is $N$. Thus $(I_X)_3=k\cdot LMN$.

By Burch's theorem, since $X$ lies on the cubic $LMN$, one has $\mu(I_X)\leq 4$. If
$\mu(I_X)\leq 3$, then $I_X$ is either a complete intersection or a strict almost
complete intersection, and Remark~\ref{rem:ci-linear-type} gives $\rt(X)=1$. If
$\mu(I_X)=4$, Lemma~\ref{lem:cubic-four-generated-rt} gives $\rt(X)=3$.
\end{proof}

\begin{Proposition}\label{prop:r6-conic-residual}
Let $X=Y\cup Z$ be an $(s-6)$-fold collinear configuration with $a=s-6\geq 4$.
Assume that $Z$ lies on a conic and that no four points of $Z$ are collinear. Then
$\rt(X)=1$.
\end{Proposition}

\begin{proof}
Under these assumptions, $I_Z$ is a complete intersection of type $(2,3)$. Write
$I_Z=(q,c)$ with $\deg q=2$ and $\deg c=3$. After a projective change of coordinates,
assume $L=V(z)$. Since $Z\cap L=\emptyset$, the images
$\overline q,\overline c\in k[x,y]$ have no common zero on $L\simeq\PP^1$. Hence
they form a regular sequence of degrees $2$ and $3$ in $k[x,y]$, and therefore
\[
(\overline q,\overline c)_j=k[x,y]_j\qquad\text{for all }j\geq 4.
\]
Let $g\in k[x,y]_a$ be the reduced equation of $Y$ on $L$. Since $a\geq 4$, there is
$F\in(I_Z)_a$ whose image modulo $z$ is $g$. By Lemma~\ref{lem:bdl-s-r},
\[
I_X=zI_Z+(F)=(zq,zc,F).
\]
Thus $I_X$ is generated by at most three elements. Hence
Remark~\ref{rem:ci-linear-type} gives $\rt(X)=1$.
\end{proof}

\begin{Proposition}\label{prop:r6-generic-residual}
Let $X=Y\cup Z$ be an $(s-6)$-fold collinear configuration with $a=s-6\geq 4$.
Assume that $Z$ lies on no conic. Then $\rt(X)=3$.
\end{Proposition}

\begin{proof}
Since $|Z|=6=\binom{4}{2}$ and $Z$ lies on no conic, we have $(I_Z)_2=0$. Thus $Z$ is
a $3$-triangular residual block in generic position. Since $a\geq 4$, the result follows
from Theorem~\ref{thm:rt-triangular-block}.
\end{proof}

Combining the preceding results gives the classification for $s\geq 10$.

\begin{Proposition}\label{prop:r6-classification-s-geq10}
Let $X=Y\cup Z$ be an $(s-6)$-fold collinear configuration with $s\geq 10$. Then:
\begin{enumerate}
\item[\rm (i)] $\rt(X)=1$ if and only if one of the following holds:
\begin{enumerate}
\item[\rm (a)] the six points of $Z$ are collinear;
\item[\rm (b)] exactly five points of $Z$ are collinear on a line $M$, and, with
$P=L\cap M$, one has $s=11$ and $P\notin Y$;
\item[\rm (c)] exactly four points of $Z$ are collinear and $\mu(I_X)\leq 3$;
\item[\rm (d)] no four points of $Z$ are collinear and $Z$ lies on a conic.
\end{enumerate}
\item[\rm (ii)] $\rt(X)=3$ if and only if one of the following holds:
\begin{enumerate}
\item[\rm (a)] $Z$ lies on no conic;
\item[\rm (b)] exactly five points of $Z$ are collinear on a line $M$, and either
$s\neq 11$, or $s=11$ and $L\cap M\in Y$;
\item[\rm (c)] exactly four points of $Z$ are collinear and $\mu(I_X)=4$.
\end{enumerate}
\end{enumerate}
\end{Proposition}

\begin{proof}
The result follows from Propositions~\ref{prop:r6-collinear-five-plus-one},
\ref{prop:r6-four-collinear}, \ref{prop:r6-conic-residual}, and
\ref{prop:r6-generic-residual}. In the case of exactly four collinear points,
Proposition~\ref{prop:r6-four-collinear} gives $\mu(I_X)\leq 4$, so the alternatives
$\mu(I_X)\leq 3$ and $\mu(I_X)=4$ exhaust all possibilities.
\end{proof}

It remains to treat the boundary case $s=9$, equivalently $a=3$. After a projective
change of coordinates, assume $L=V(z)$, and let $g\in k[x,y]_3$ be the reduced equation
of $Y$ on $L$.

\begin{Lemma}\label{lem:nine-points-maximal-HF}
Let $W\subseteq \PP^2_k$ be a reduced set of nine points with Hilbert function
\[
\HF_{R/I_W}=1,\ 3,\ 6,\ 9,\ 9,\ldots .
\]
Then $\rt(W)=3$.
\end{Lemma}

\begin{proof}
The Hilbert function gives one cubic generator and three new quartic generators. Thus
$I_W$ has a Hilbert--Burch resolution
\[
0\lar R(-5)^3\lar R(-3)\oplus R(-4)^3\lar I_W\lar 0.
\]
The determinant of the $3\times 3$ block of linear entries is the cubic generator and
is nonzero. Hence, as in Lemma~\ref{lem:cubic-four-generated-rt}, Morey's criterion
gives $\rt(W)\leq 3$. Since $\mu(I_W)=4$, the ideal is not of linear type, and
Proposition~\ref{prop:points-plane-gap} gives $\rt(W)=3$.
\end{proof}

\begin{Proposition}\label{prop:r6-s9-classification}
Let $X=Y\cup Z$ be an $(s-6)$-fold collinear configuration with $s=9$. Then:
\begin{enumerate}
\item[\rm (i)] if the six points of $Z$ are collinear, then $\rt(X)=1$;
\item[\rm (ii)] if exactly five points of $Z$ are collinear, then $\rt(X)=3$;
\item[\rm (iii)] if exactly four points of $Z$ are collinear on a line $M$, then, with
$P=L\cap M$, the following hold:
\begin{enumerate}
\item[\rm (a)] if $P\notin Y$, then $\rt(X)=3$;
\item[\rm (b)] if $P\in Y$ and
$W:=(Y\setminus\{P\})\cup (Z\setminus M)$, then $\rt(X)=1$ if no three points of $W$
are collinear, and $\rt(X)=3$ if exactly three points of $W$ are collinear.
\end{enumerate}
\item[\rm (iv)] assume that no four points of $Z$ are collinear and that $Z$ lies on a
conic. Write $I_Z=(q,c)$, with $\deg q=2$ and $\deg c=3$. Then
\[
\rt(X)=1
\Longleftrightarrow
g\in(\overline q,\overline c)_3\subseteq k[x,y]_3.
\]
If $g\notin(\overline q,\overline c)_3$, then $\rt(X)=3$.
\item[\rm (v)] if $Z$ lies on no conic, then $\rt(X)=3$.
\end{enumerate}
\end{Proposition}

\begin{proof}
Parts {\rm (i)} and {\rm (ii)} follow from
Proposition~\ref{prop:r6-collinear-five-plus-one}. Part {\rm (iii)} follows from
Proposition~\ref{prop:r6-four-collinear}.

For {\rm (iv)}, since no four points of $Z$ are collinear and $Z$ lies on a conic,
$I_Z=(q,c)$ is a complete intersection of type $(2,3)$. If
$g\in(\overline q,\overline c)_3$, there exists $F\in(I_Z)_3$ whose image modulo $z$
is $g$, and Lemma~\ref{lem:bdl-s-r} gives
\[
I_X=zI_Z+(F)=(zq,zc,F).
\]
Hence $\rt(X)=1$ by Remark~\ref{rem:ci-linear-type}.

Conversely, assume that $g\notin(\overline q,\overline c)_3$. If $H\in(I_X)_3$, then
$\overline H$ vanishes on $Y$, so $\overline H$ is a scalar multiple of $g$. Since
$H\in(I_Z)_3$, one also has $\overline H\in(\overline q,\overline c)_3$, and hence
$\overline H=0$. Thus $H=zQ$ with $Q\in(I_Z)_2=(q)$, so
$(I_X)_3=k\cdot zq$. There is no quadric through $X$, and therefore
\[
\HF_{R/I_X}=1,\ 3,\ 6,\ 9,\ 9,\ldots .
\]
Lemma~\ref{lem:nine-points-maximal-HF} gives $\rt(X)=3$.

For {\rm (v)}, since $Z$ lies on no conic, $(I_Z)_2=0$ and the reduction map
$(I_Z)_3\lar k[x,y]_3$ is an isomorphism. Hence there exists a unique
$F\in(I_Z)_3$ whose image modulo $z$ is $g$, and $I_X=zI_Z+(F)$. Thus the cubic part
of $I_X$ is one-dimensional, and again
$\HF_{R/I_X}=1,\ 3,\ 6,\ 9,\ 9,\ldots$.
Lemma~\ref{lem:nine-points-maximal-HF} gives $\rt(X)=3$.
\end{proof}

\begin{Remark}\label{rem:r6-s9-conic-condition}
In Proposition~\ref{prop:r6-s9-classification}\textup{(iv)}, the condition
$g\in(\overline q,\overline c)_3$ means that there exists a cubic through the six
points of $Z$ whose restriction to $L$ cuts out the divisor $Y$. Equivalently, there is
a cubic through all nine points of $X$ which is not forced to contain $L$ as a
component. If no such cubic exists, then the only cubic through $X$ is $zq$, and
$X$ has maximal Hilbert function $1,3,6,9,9,\ldots$.
\end{Remark}
The preceding results complete the classification for residual blocks of cardinality $r\leq 6$. The next residual size, $r=7$, is qualitatively different: cubic equations through $Z$ enter essentially, and the relation type is no longer determined by the line-and-conic incidence data alone.

\section{Configurations with Small Number of Points}
\label{sec:small-number-points}

In this section we apply the preceding results to finite sets of small cardinality. The
main point is that, up to ten points, relation type has only the values $1$ and $3$.

\begin{Theorem}\label{thm:rt-up-to-ten-points}
Let $X\subseteq \PP^2_k$ be a finite set of $s$ distinct points with $4\leq s\leq 10$.
Then $\rt(X)\in\{1,3\}$. More precisely:
\begin{enumerate}
\item[\rm (i)] If $s=4$ or $s=5$, then $\rt(X)=1$.

\item[\rm (ii)] If $s=6$, then
\[
\rt(X)=
\begin{cases}
1,& \text{if }X\text{ is contained in a conic},\\
3,& \text{otherwise}.
\end{cases}
\]

\item[\rm (iii)] If $s=7$, call $X$ a \emph{four-plus-three configuration} if exactly
four points of $X$ are collinear and the remaining three are not collinear. Then
\[
\rt(X)=
\begin{cases}
3,& \text{if }X\text{ is a four-plus-three configuration},\\
1,& \text{otherwise}.
\end{cases}
\]

\item[\rm (iv)] If $s=8$, then $\rt(X)=3$ exactly in the following cases:
\begin{enumerate}
\item[\rm (a)] exactly five points of $X$ are collinear and the remaining three are not
collinear;
\item[\rm (b)] exactly four points of $X$ are collinear, and among the remaining four
points exactly three are collinear;
\item[\rm (c)] no four points of $X$ are collinear, and seven, but not all eight, points
of $X$ lie on a conic.
\end{enumerate}
In all other cases, $\rt(X)=1$.

\item[\rm (v)] If $s=9$, then $\rt(X)=3$ exactly in the following cases:
\begin{enumerate}
\item[\rm (a)] $X$ has maximal Hilbert function
\[
\HF_{R/I_X}=1,\ 3,\ 6,\ 9,\ 9,\ldots;
\]
\item[\rm (b)] exactly six points of $X$ are collinear and the remaining three are not
collinear;
\item[\rm (c)] exactly five points of $X$ are collinear, and among the remaining four
points exactly three are collinear.
\end{enumerate}
In all other cases, $\rt(X)=1$.

\item[\rm (vi)] If $s=10$, then $\rt(X)=3$ exactly in the following cases:
\begin{enumerate}
\item[\rm (a)] exactly seven points of $X$ are collinear and the remaining three are not
collinear;
\item[\rm (b)] exactly six points of $X$ are collinear, and among the remaining four
points exactly three are collinear;
\item[\rm (c)] exactly five points of $X$ are collinear and the remaining five points are
not collinear;
\item[\rm (d)] $X$ has maximal Hilbert function
\[
\HF_{R/I_X}=1,\ 3,\ 6,\ 10,\ 10,\ldots;
\]
\item[\rm (e)] no four points of $X$ are collinear, and nine, but not all ten, points of
$X$ lie on a conic.
\end{enumerate}
In all other cases, $\rt(X)=1$.
\end{enumerate}
\end{Theorem}

\begin{proof}
If $s=4$ or $s=5$, then $X$ is contained in a conic, and
Proposition~\ref{prop:points-on-conic} gives $\rt(X)=1$.

\medskip
\noindent{\bf Case $s=6$.}. If $X$ is contained in a conic, then $\rt(X)=1$. Assume that $X$ is not
contained in a conic. If three points of $X$ are collinear, then the remaining three
points are not collinear; otherwise $X$ would lie on a reducible conic. Hence
Corollary~\ref{cor:rt-s-3-fold} gives $\rt(X)=3$. If no three points are collinear, then
$(I_X)_2=0$, so $X$ consists of $6=\binom{4}{2}$ points in generic position. Thus
$I_X$ has a linear Hilbert--Burch resolution
\[
0\lar R(-4)^3\lar R(-3)^4\lar I_X\lar 0,
\]
and Proposition~\ref{prop:points-plane-gap} gives $\rt(X)=3$.

\medskip
\noindent{\bf Case $s=7$.}
If $X$ is a four-plus-three configuration, then
Corollary~\ref{cor:rt-s-3-fold} gives $\rt(X)=3$. Conversely, assume that $X$ is not a
four-plus-three configuration. If $X$ is contained in a conic, then $\rt(X)=1$. Hence
assume that $X$ is not contained in a conic. If at least five points are collinear, then
$X$ is an $(s-r)$-fold collinear configuration with $r\leq 2$, so
Proposition~\ref{prop:r012} gives $\rt(X)=1$. If exactly four points are collinear, then
the remaining three points must be collinear, because $X$ is not a four-plus-three
configuration; hence $X$ lies on a reducible conic, a contradiction. Thus no four points
are collinear.

Therefore $(I_X)_2=0$ and
\[
\HF_{R/I_X}=1,\ 3,\ 6,\ 7,\ 7,\ldots .
\]
Hence $\dim_k(I_X)_3=3$. Let $J$ be the ideal generated by $(I_X)_3$. The cubic system
has no fixed component: a fixed conic would contain $X$, while a fixed line would leave
a system of conics through at least four points, of vector-space dimension at most $2$,
contradicting $\dim_k(I_X)_3=3$. Thus $\hht J=2$. Since $J$ is generated by three
cubics, its Hilbert--Burch matrix has column degrees $1$ and $2$, and so
\[
\deg(R/J)=1^2+1\cdot 2+2^2=7.
\]
As $J\subseteq I_X$ and both ideals define schemes of degree $7$, we get $J=I_X$.
Thus $I_X$ is a strict almost complete intersection, and
Remark~\ref{rem:ci-linear-type} gives $\rt(X)=1$.

\medskip
\noindent{\bf Case $s=8$.}
 If $X$ is contained in a conic, then $\rt(X)=1$. Assume that $X$ is not
contained in a conic. Then at most five points are collinear. If exactly five points are
collinear, the residual three points are not collinear, and
Corollary~\ref{cor:rt-s-3-fold} gives $\rt(X)=3$.

If exactly four points are collinear, Proposition~\ref{prop:r4-classification} applies:
it gives $\rt(X)=3$ precisely when the residual four-point block has exactly three
collinear points, and gives $\rt(X)=1$ in the remaining residual cases.

It remains to assume that no four points are collinear. Then
\[
\HF_{R/I_X}=1,\ 3,\ 6,\ 8,\ 8,\ldots ,
\]
so $\dim_k(I_X)_3=2$ and $\dim_k(I_X)_4=7$. Suppose first that seven points lie on a
conic $C=V(q)$, and let $p$ be the remaining point. Since no four points are collinear,
$C$ is irreducible, and
\[
(I_X)_3=q(I_p)_1.
\]
Thus $\dim_k R_1(I_X)_3=5$, while $\dim_k(I_X)_4=7$. Hence $I_X$ has two cubic and two
quartic minimal generators, and its Hilbert--Burch resolution has the form
\[
0\lar R(-4)\oplus R(-5)^2
\lar R(-3)^2\oplus R(-4)^2
\lar I_X\lar 0.
\]
The linear syzygy between the two cubic generators, together with the linear entries in
the remaining columns, gives a Jacobian-dual minor nonzero modulo $\fm S$. Morey's
criterion \cite[Proposition~3.1]{Morey} gives $\rt(X)\leq 3$. Since $\mu(I_X)=4$,
$I_X$ is not of linear type, and Proposition~\ref{prop:points-plane-gap} gives
$\rt(X)=3$.

If no seven points lie on a conic, then the pencil $(I_X)_3$ has no fixed component.
Thus its two cubic generators have no common component, and
$\dim_k R_1(I_X)_3=6$. Since $\dim_k(I_X)_4=7$, exactly one quartic generator is needed.
Hence $I_X$ is generated by two cubics and one quartic, so it is a strict almost
complete intersection. Therefore $\rt(X)=1$.

\medskip
\noindent{\bf Case $s=9$.}
 If $X$ is contained in a conic, then $\rt(X)=1$. Assume that $X$ is not
contained in a conic. If seven or more points are collinear, then
Proposition~\ref{prop:r012} gives $\rt(X)=1$. If exactly six points are collinear, the
residual three points are not collinear, and Corollary~\ref{cor:rt-s-3-fold} gives
$\rt(X)=3$.

Assume next that exactly five points are collinear, say $X=Y\cup Z$ with $|Y|=5$ and
$|Z|=4$. The residual block $Z$ is not collinear. If no three points of $Z$ are
collinear, then Proposition~\ref{prop:r4-classification} gives $\rt(X)=1$. If exactly
three points of $Z$ are collinear, then the same proposition gives $\rt(X)=3$, since
$s=9>7$.

Now suppose no line contains five points, but some line contains exactly four points.
Write $X=Y\cup Z$ with $|Y|=4$ and $|Z|=5$. If no three points of $Z$ are collinear,
then every cubic through $X$ contains the line through $Y$, and the residual conic
through $Z$ is unique. Hence $\dim_k(I_X)_3=1$, so $X$ has maximal Hilbert function;
Lemma~\ref{lem:nine-points-maximal-HF} gives $\rt(X)=3$. If $Z$ contains four
collinear points, then, since no line contains five points, Proposition~\ref{prop:r5-classification}
gives $\rt(X)=1$. If $Z$ contains exactly three collinear points, choose that line as
the distinguished line. Then Proposition~\ref{prop:r6-s9-classification} applies; by
inspection of its alternatives, the relation-type-three cases are precisely the
maximal-Hilbert-function cases, and otherwise $\rt(X)=1$.

The same argument applies if the maximum number of collinear points in $X$ is three:
choose a collinear triple as distinguished line and apply
Proposition~\ref{prop:r6-s9-classification}. Again the relation-type-three cases are
exactly the maximal-Hilbert-function cases.

Finally assume that no three points are collinear. Since $X$ is not contained in a
conic, $(I_X)_2=0$. If $\dim_k(I_X)_3=1$, then $X$ has maximal Hilbert function, and
Lemma~\ref{lem:nine-points-maximal-HF} gives $\rt(X)=3$. If
$\dim_k(I_X)_3\geq 2$, choose two independent cubics through $X$. If they have no
common component, they form a complete intersection of type $(3,3)$ containing the
nine reduced points of $X$, hence equal to $X$; thus $\rt(X)=1$. If they have a common
component, it cannot be a line, since no three points are collinear. Hence eight points
of $X$ lie on a conic and the ninth lies off it. Then $I_X$ is generated in degrees
$3,3,4$, hence is a strict almost complete intersection. Thus $\rt(X)=1$.

 \medskip
\noindent{\bf Case $s=10$.}
 If $X$ is contained in a conic, then $\rt(X)=1$. Assume that $X$ is
not contained in a conic. If at least eight points are collinear, then $X$ lies on a
reducible conic, impossible. If exactly seven points are collinear, the residual three
points are not collinear, and Corollary~\ref{cor:rt-s-3-fold} gives $\rt(X)=3$.

If exactly six points are collinear, write $X=Y\cup Z$ with $|Y|=6$ and $|Z|=4$.
Proposition~\ref{prop:r4-classification} gives $\rt(X)=3$ precisely when $Z$ has
exactly three collinear points; otherwise $\rt(X)=1$.

If exactly five points are collinear, write $X=Y\cup Z$ with $|Y|=5$ and $|Z|=5$.
Since $X$ is not contained in a conic, $Z$ is not collinear. If four points of $Z$ are
collinear, Proposition~\ref{prop:r5-classification} gives $\rt(X)=3$.
If no four points of $Z$ are collinear, then $Z$ lies on a unique conic. If this conic
meets the distinguished line in two distinct points, Proposition~\ref{prop:r5-classification}
gives $\rt(X)=3$. If the intersection is non-reduced, the same proof applies after
replacing the normal form $\overline q=xy$ by $\overline q=x^2$. Hence $\rt(X)=3$ in
all cases where exactly five, but not all five residual points, are collinear.

Now assume that no line contains five points, but some line contains exactly four
points. Write $X=Y\cup Z$, where $|Y|=4$ and $Y$ lies on a line $L$. If $Z$ lies on no
conic, then every cubic through $X$ contains $L$, and the residual conic would have to
contain $Z$. Hence $(I_X)_3=0$, so $X$ has maximal Hilbert function. Therefore
Example~\ref{ex:triangular-number-general-points} and
Proposition~\ref{prop:points-plane-gap} give $\rt(X)=3$.

Assume that $Z$ lies on a conic. If no four points of $Z$ are collinear, then
Proposition~\ref{prop:r6-conic-residual} gives $\rt(X)=1$. If exactly four points of
$Z$ are collinear, then $X$ has no quadric equation and has a unique cubic equation,
the product of the two four-point lines and the line through the remaining two points.
Thus
\[
\HF_{R/I_X}=1,\ 3,\ 6,\ 9,\ 10,\ 10,\ldots .
\]
Equivalently, the Hilbert series gives a resolution
\[
0\lar R(-5)\oplus R(-6)
\lar R(-3)\oplus R(-4)^2
\lar I_X\lar 0.
\]
Hence $\mu(I_X)=3$, and Remark~\ref{rem:ci-linear-type} gives $\rt(X)=1$.

It remains to assume that no four points are collinear. Since $X$ is not contained in a
conic, $(I_X)_2=0$. If $(I_X)_3=0$, then $X$ has maximal Hilbert function
\[
\HF_{R/I_X}=1,\ 3,\ 6,\ 10,\ 10,\ldots .
\]
Since $10=\binom{5}{2}$, Example~\ref{ex:triangular-number-general-points} gives a
linear Hilbert--Burch resolution
\[
0\lar R(-5)^4\lar R(-4)^5\lar I_X\lar 0,
\]
and Proposition~\ref{prop:points-plane-gap} gives $\rt(X)=3$.

If $\dim_k(I_X)_3=1$, then
\[
\HF_{R/I_X}=1,\ 3,\ 6,\ 9,\ 10,\ 10,\ldots .
\]
Thus the Hilbert series gives
\[
\mathrm{HS}_{R/I_X}(t)=\frac{1-t^3-2t^4+t^5+t^6}{(1-t)^3}.
\]
Hence $I_X$ has one cubic generator and two quartic generators, with Hilbert--Burch
resolution
\[
0\lar R(-5)\oplus R(-6)
\lar R(-3)\oplus R(-4)^2
\lar I_X\lar 0.
\]
Thus $I_X$ is a strict almost complete intersection, and $\rt(X)=1$.

Finally assume that $\dim_k(I_X)_3\geq 2$. Any two independent cubics through $X$ have
a common component, since a complete intersection of two cubics has degree $9$. The
common component cannot be a line, because no four points are collinear; hence it is a
conic. Since $X$ is not contained in a conic, exactly nine points lie on this conic and
the tenth point lies off it. Let $C=V(q)$ be the conic and let $p$ be the remaining
point, with $I_p=(\ell_1,\ell_2)$. Then
\[
(I_X)_3=q(I_p)_1.
\]
Moreover every quartic through $X$ is divisible by $q$, since its restriction to
$C\simeq\PP^1$ has degree $8$ and vanishes at the nine points of $X\cap C$. Hence
\[
(I_X)_4=q(I_p)_2=R_1(I_X)_3.
\]
Thus there are no new quartic generators. In degree $5$,
$\dim_k(I_X)_5=21-10=11$, while
$\dim_k R_2(I_X)_3=\dim_k q(I_p)_3=9$; hence two quintic generators are needed. Therefore
$I_X$ has two cubic and two quintic minimal generators, and the Hilbert series gives
\[
0\lar R(-4)\oplus R(-6)^2
\lar R(-3)^2\oplus R(-5)^2
\lar I_X\lar 0.
\]
In particular $\mu(I_X)=4$, so $I_X$ is not of linear type. The linear syzygy between
$q\ell_1$ and $q\ell_2$, together with the linear entries in the remaining columns,
gives a Jacobian-dual minor nonzero modulo $\fm S$. Hence Morey's criterion gives
$\rt(X)\leq 3$. Since relation type $2$ does not occur for ideals of points in $\PP^2_k$,
we get $\rt(X)=3$.

The cases are exhaustive, and the theorem follows.
\end{proof}

\section{The First Higher Values and Final Questions}
\label{sec:first-higher-values}

Theorem~\ref{thm:rt-up-to-ten-points} shows that, for finite sets of at most ten points
in $\PP^2_k$, the relation type is always either $1$ or $3$. In this final section we
show that the first higher value occurs for eleven points in generic position. We then
explain why relation type $4$ is not a global gap, and we close with some questions.

We first recall two ingredients. Let $I=I_X$ be the defining ideal of a finite reduced
set of points in $\PP^2_k$, let $S=R[T_0,\ldots,T_m]$, and let $\Q\subseteq S$ be the
defining ideal of $\Rees_R(I)$. If $\mathcal L=\Q\langle 1\rangle$ is the ideal of
linear equations of the symmetric algebra, then, as explained in
Section~\ref{sec:rt-geometry-points},
\[
\Q=\mathcal L:_S(\fm S)^\infty .
\]
Thus, for point ideals, the nonlinear equations of the Rees algebra are obtained by
saturating the symmetric equations with respect to $\fm S$.

We shall also use Jouanolou's terminology of \emph{formes d'inertie}. Let
$B=k[T_0,T_1,T_2,T_3]$ and $S=B[x,y,z]$. Suppose that
\[
\mathcal L=(L_1,L_2,Q)\subseteq S,
\]
where $L_1,L_2$ have bidegree $(1,1)$ and $Q$ has bidegree $(2,1)$ with respect to the
$(x,y,z)$-degree and the $T$-degree. Then
$\mathcal L:(x,y,z)^\infty$ is the ideal of inertia forms of the homogeneous system
$L_1,L_2,Q$. In \cite[\S3.11.19]{JouanolouInertia}, Jouanolou constructs these inertia
forms and relates the final one to the resultant. For three homogeneous forms of degrees
$d_1,d_2,d_3$, the resultant has total degree
$d_1d_2+d_1d_3+d_2d_3$ in the coefficients; see
\cite[3.11.19.25]{JouanolouInertia}. In particular, for the bidegree pattern
$(1,1,2)$ considered below, the inertia-form ideal
$\mathcal L:(x,y,z)^\infty$ has no minimal generator of $T$-degree larger than $5$.
The relation with torsion in the symmetric algebra and implicitization is also explained
in \cite{BCJ}.

\begin{Theorem}\label{thm:eleven-generic-position-rt-five}
Let $X\subseteq \PP^2_k$ be a set of eleven points in generic position. Then
$\rt(X)=5$.
\end{Theorem}

\begin{proof}
Since $X$ is in generic position,
\[
\HF_{R/I_X}=1,\ 3,\ 6,\ 10,\ 11,\ 11,\ldots .
\]
Thus $(I_X)_3=0$ and $\dim_k(I_X)_4=15-11=4$. Hence
\[
\mathrm{HS}_{R/I_X}(t)=\frac{1-4t^4+2t^5+t^6}{(1-t)^3}.
\]
Since $I_X$ is saturated of height two, it is perfect, and Hilbert--Burch gives
\[
0\lar R(-5)^2\oplus R(-6)\lar R(-4)^4\lar I_X\lar 0.
\]
Thus $I_X=(f_0,f_1,f_2,f_3)$ is generated by four quartics, and its Hilbert--Burch
column degrees are $(1,1,2)$.

We first prove the lower bound. The four quartics have no common factor, since $I_X$
has height two. Hence the linear system $(I_X)_4$ has no fixed component, and the
associated rational map
\[
\phi:\PP^2_k\dashrightarrow \PP^3_k
\]
is generically finite onto a surface $Y\subseteq\PP^3_k$. Since $X$ is reduced, $I_X$
is locally a complete intersection on the punctured spectrum, and hence satisfies $G_3$.
By the special-fiber degree formula \cite[Corollary~C]{CidRuiz},
\[
\deg(\phi)\deg_{\PP^3}(Y)=e_2(1,1,2)=1\cdot 1+1\cdot 2+1\cdot 2=5.
\]
The four quartics are linearly independent, so $Y$ is not contained in a plane. Since
$5$ is prime, it follows that
\[
\deg(\phi)=1,\qquad \deg_{\PP^3}(Y)=5.
\]
Let $B=k[T_0,T_1,T_2,T_3]$, and let $\mathcal K$ be the defining ideal of the special
fiber $\mathcal F(I_X)\simeq B/\mathcal K$. Since $Y\subseteq\PP^3_k$ is a surface,
$\mathcal K$ is a height-one prime ideal of the UFD $B$, hence
$\mathcal K=(E)$ for an irreducible form $E$ of degree $5$. Since the special fiber is
obtained from the Rees algebra by reducing modulo $\fm=(x,y,z)$, the equation $E$ lifts
to an equation of the Rees algebra of $T$-degree $5$. Hence $\rt(X)\geq 5$.

For the upper bound, let $\mathcal L$ be the ideal of linear equations of
$\operatorname{Sym}_R(I_X)$. The Hilbert--Burch column degrees $(1,1,2)$ give
\[
\mathcal L=(L_1,L_2,H),
\]
where $L_1,L_2$ have bidegree $(1,1)$ and $H$ has bidegree $(2,1)$. Since $X$ is a
finite reduced set of points,
\[
\Q=\mathcal L:_S(\fm S)^\infty .
\]
By the inertia-form bound recalled above, this saturation has no minimal generator of
$T$-degree larger than $5$. Thus $\rt(X)\leq 5$, and therefore $\rt(X)=5$.
\end{proof}

Together with Theorem~\ref{thm:rt-up-to-ten-points},
Theorem~\ref{thm:eleven-generic-position-rt-five} shows that relation type larger than
$3$ first appears at cardinality eleven: if $|X|\leq 10$, then
$\rt(X)\in\{1,3\}$, while eleven points in generic position have relation type $5$.

\begin{Question}\label{ques:eleven-points-spectrum}
Determine the relation-type spectrum of finite sets of eleven points in $\PP^2_k$.
Theorem~\ref{thm:eleven-generic-position-rt-five} shows that $5$ occurs. Is it true that
\[
\{\,\rt(X)\mid X\subseteq \PP^2_k,\ |X|=11\,\}\subseteq \{1,3,5\}?
\]
\end{Question}

The restriction to eleven points in Question~\ref{ques:eleven-points-spectrum} is
essential. Although relation type $4$ does not occur for at most ten points, it is not
a global gap. The special fiber gives a simple numerical mechanism by which relation
type $4$ can appear.

Let $I=(f_0,f_1,f_2,f_3)\subseteq R$ be an equigenerated height-two perfect ideal with
Hilbert--Burch column degrees $(d_1,d_2,d_3)$, and assume that the associated rational
map
\[
\phi:\PP^2_k\dashrightarrow \PP^3_k
\]
is generically finite onto its image $Y$. For point ideals, the special-fiber degree
formula \cite[Corollary~C]{CidRuiz} gives
\[
\deg(\phi)\deg(Y)=d_1d_2+d_1d_3+d_2d_3.
\]
For eleven points in generic position the column degrees are $(1,1,2)$, and the right
side is $5$. By contrast, the column degrees $(1,2,2)$ give
\[
d_1d_2+d_1d_3+d_2d_3=1\cdot2+1\cdot2+2\cdot2=8.
\]
If in such a case $\deg(\phi)=2$, then the image surface has degree $4$. Thus the
special fiber has a quartic implicit equation, and this gives a Rees equation of
$T$-degree $4$.

\begin{Example}\label{ex:seventeen-points-rt-four}
There exists a reduced set $X\subseteq \PP^2_k$ of $17$ points with relation type $4$.
Work over $k=\Q$, and let
\[
\begin{aligned}
X=\{&
[\pm 1:0:1],\ [\pm 2:1:1],\ [\pm 3:2:1],\ [\pm 4:3:1],\\
&[\pm 5:5:1],\ [\pm 6:8:1],\ [\pm 7:13:1],\ [\pm 8:21:1],\ [1:0:0]\}.
\end{aligned}
\]
A direct computation with \textsc{Singular} shows that $I_X$ is generated by four
quintics and has Hilbert--Burch resolution
\[
0\lar R(-6)\oplus R(-7)^2
\lar R(-5)^4
\lar I_X
\lar 0.
\]
Equivalently, the Hilbert--Burch column degrees are $(1,2,2)$. The same computation
shows that the rational map $\PP^2_k\dashrightarrow\PP^3_k$ defined by these four
quintics has degree $2$ onto its image, and the image is a quartic surface. Hence the
defining ideal of the special fiber has an essential equation of degree $4$. Moreover,
the defining ideal of the Rees algebra has no minimal equation of $T$-degree larger
than $4$. Therefore
\[
\rt(X)=4.
\]
\end{Example}

\begin{Remark}\label{rem:rt-four-not-global-gap}
Example~\ref{ex:seventeen-points-rt-four} shows that relation type $4$ cannot be
excluded globally. The phenomenon in small cardinalities is subtler: relation type $4$
does not occur for sets of at most ten points, while the first higher value is $5$,
realized by eleven points in generic position. The remaining problem is to understand
whether relation type $4$ occurs for eleven points.
\end{Remark}

We finish by recalling a family of configurations which produces further large relation
types and illustrates the role of the special fiber. Let $d=2h$ with $h\geq 1$. Choose
two distinct points $p,q\in\PP^2_k$ and two pencils of lines
\[
L_1,\ldots,L_d \quad \text{through }p,
\qquad
L'_1,\ldots,L'_d \quad \text{through }q,
\]
so that all intersection points considered below are distinct. The Geramita--Maroscia
basic configuration of rank $d$ is
\[
B_d=
\bigcup_{i=1}^{h}
\Bigl((L_{2i-1}\cup L_{2i})\cap (L'_1\cup\cdots\cup L'_{2i})\Bigr).
\]
Equivalently,
\[
B_d=\{\,L_j\cap L'_t \mid 1\leq j\leq d,\ 1\leq t\leq m_j\,\},
\qquad
(m_1,\ldots,m_d)=(2,2,4,4,\ldots,d,d).
\]
Thus
\[
|B_d|=2+2+4+4+\cdots+d+d=2h(h+1)=\frac{d(d+2)}{2}.
\]

Set
\[
\alpha_i=L_{2i-1}L_{2i},
\qquad
\beta_i=L'_{2i-1}L'_{2i},
\qquad i=1,\ldots,h.
\]
Define
\[
F_i:=\beta_1\cdots\beta_{i-1}\alpha_i\cdots\alpha_h,
\qquad i=1,\ldots,h+1.
\]
Geramita and Maroscia show that $B_d$ has generic Hilbert function, lies on no curve of
degree $<d$, and that $I_{B_d}$ is generated in degree $d$; see
\cite[Lemma~2.1 and Corollary~2.5]{GM}. More precisely,
\[
I_{B_d}=(F_1,\ldots,F_{h+1}),
\]
and $I_{B_d}$ has minimal free resolution
\[
0\lar R(-d-2)^h
\stackrel{\Psi}{\lar}
R(-d)^{h+1}
\lar I_{B_d}
\lar 0,
\]
where
\[
\Psi=
\begin{pmatrix}
\beta_1 & 0 & \cdots & 0\\
-\alpha_1 & \beta_2 & \ddots & \vdots\\
0 & -\alpha_2 & \ddots & 0\\
\vdots & \ddots & \ddots & \beta_h\\
0 & \cdots & 0 & -\alpha_h
\end{pmatrix}.
\]
Thus $I_{B_d}$ is equigenerated, but its Hilbert--Burch matrix has quadratic entries.

For $d=2$, the ideal is a complete intersection. For $d=4$, it is generated by three
quartics and is a strict almost complete intersection. Hence in both cases
$\rt(B_d)=1$. For $d=2h\geq 6$, one has $h+1\geq 4$, so $I_{B_d}$ is not of linear type,
and therefore $\rt(B_d)\geq 3$.

Let
\[
\phi_d:\PP^2_k\dashrightarrow \PP^h_k
\]
be the rational map defined by the minimal generators of $I_{B_d}$, and let $V_d$ be
the closure of its image. If $\phi_d$ is generically finite of degree $e$, then the
special-fiber degree formula gives
\[
e\deg V_d=e_2(2,\ldots,2)=4\binom{h}{2}=2h(h-1)=\frac{d(d-2)}{2}.
\]
In particular, for $d=6$ one has $h=3$, $|B_6|=24$, and
\[
e\deg V_6=12.
\]
For the standard configuration $B_6=(2,2,4,4,6,6)$, computations with
\textsc{Singular} show that the associated map is birational onto its image. Hence
$V_6\subseteq\PP^3_k$ is a surface of degree $12$, and its implicit equation gives an
essential equation of the special fiber of $T$-degree $12$. Therefore
$\rt(B_6)\geq 12$. The same computation shows that the defining ideal of the Rees
algebra has no minimal generator of $T$-degree larger than $12$. Hence
\[
\rt(B_6)=12.
\]

\begin{Question}\label{ques:final-spectrum}
For each $s\geq 1$, determine the set
\[
\mathcal R_s:=\{\,\rt(X)\mid X\subseteq \PP^2_k
\text{ is a reduced set of }s\text{ points}\,\}.
\]
We have shown that $\mathcal R_s\subseteq\{1,3\}$ for $s\leq 10$, and that
$5\in\mathcal R_{11}$. Computations show that $4\in\mathcal R_{17}$. How does
$\mathcal R_s$ grow with $s$?
\end{Question}

\begin{Question}\label{ques:GM-relation-type}
Let $B_d\subseteq \PP^2_k$ be the Geramita--Maroscia basic configuration of rank
$d=2h\geq 6$. Determine $\rt(B_d)$. More precisely, how do the equations of the special
fiber interact with the torsion equations of the symmetric algebra? Is there a closed
formula for $\rt(B_d)$ in terms of $d$?
\end{Question}


\end{document}